\documentclass[11 pt]{amsart}

\usepackage{hyperref}
\usepackage{amssymb, amsmath, amsthm, amsfonts}
\usepackage{enumitem}
\usepackage{verbatim}
\usepackage{bbm}
\usepackage{enumerate}
\usepackage{dsfont}
\usepackage{upgreek}
\usepackage[mathscr]{eucal}
\usepackage{tikz}
\usepackage[normalem]{ulem}

\usepackage[backgroundcolor=black!5,linecolor=black]{todonotes}
\usepackage{comment}

\usepackage[explicit,indentafter]{titlesec}
\titleformat{\section}{\centering\normalfont\scshape}{\thesection.}{.5em}{#1}
\titleformat{\subsection}[runin]{\normalfont\itshape}{\textnormal{\thesubsection.}}{.5em}{#1.}
\titleformat{\subsubsection}[runin]{\normalfont\itshape}{\thesubsubsection.}{.5em}{#1.}
\titlespacing{\section}{0em}{1em}{0.5em}
\titlespacing{\subsection}{0em}{.5em}{0.5em}

\usepackage{stackengine}

\newcommand{\Zsep}{\mathcal Z^{\mathrm{sep}}_{k-\ell}}
\newcommand{\Zsepj}{\mathcal Z^{j}_{k-\ell}}
\newcommand{\ubar}[1]{\underline{#1}}
\newcommand {\barw}{\bar w}
\newcommand {\ow}{\bar w}

\newcommand{\ox}{{\bar x}}
\newcommand{\oy}{{\bar y}}
\newcommand{\oz}{{\bar z}}
\newcommand{\uw}{{\ubar w}}
\newcommand{\ux}{{\ubar x}}
\newcommand{\uy}{{\ubar y}}
\newcommand{\uz}{{\ubar z}}

\usepackage{color}
\definecolor{gray}{gray}{0.5}

\newcommand{\cmt}[1]{}

\newcommand{\vertiii}[1]{{\left\vert\kern-0.25ex\left\vert\kern-0.25ex\left\vert #1
\right\vert\kern-0.25ex\right\vert\kern-0.25ex\right\vert}}

\newcommand{\detail}[1]{}

\renewcommand{\d}{(2*\n+\m)}

\newcommand{\Qfourx}[1]{\d*(\d-1)/(\d*\d+\d*\m+(2*#1-1)*(\m+1))}
\newcommand{\Qfoury}[1]{(\m+1)*(\d-1)/(\d*\d+\d*\m+(2*#1-1)*(\m+1))}
\newcommand\Qthreex[1]{(\d-#1)/( \d-#1+\m+1)}
\newcommand\Qthreey[1]{(\m+1)/(\d-#1+\m+1)}

\newcommand{\definecoords}{
\def\ptsize{.1pt}
\def\QbgFillOpacity{.2}
\def\QbgFillColor{black}
\def\QbgLineOpacity{.6}
\def\QbgDrawCritSeg{1}
\def\QbgCritSegStyle{solid}
\def\QbgCritSegOpacity{\QbgLineOpacity}
\def\QbbFillOpacity{.1}
\def\QbbLineOpacity{.5}
\def\AFillOpacity{.1}
\def\AFillColor{red}
\def\ALineOpacity{.8}
\coordinate (Q1) at (0,0);
\coordinate (Q2) at ( {(2*\n-1)/(2*\n-1+\b)}, {(2*\n-1)/(2*\n-1+\b)} );
\coordinate (Q3) at ( {\Qthreex{\b}}, { \Qthreey{\b} } );
\coordinate (Q30) at ( {\Qthreex{0}}, { \Qthreey{0} } );
\coordinate (Q4b) at ( { \Qfourx{\b} }, { \Qfoury{\b} } );
\coordinate (Q4g) at ( { \Qfourx{\g} }, { \Qfoury{\g} } );

\coordinate (C1) at ( { (\Qfourx{\b}+\Qfourx{\g})/2 }, {(\Qfoury{\b}+\Qfoury{\g})/2} );
\coordinate (C2) at (Q4b);
}

\newcommand{\drawauxlines}{
\draw (0,0) [->] -- (0,1) node [left] {$\frac1q$};
\draw (0,0) [->] -- (1,0) node [below] {$\frac1p$};
\draw [dashed,opacity=.3] (1,0) -- (0,1);
\draw [dashed,opacity=.3] (0,0) -- (1,{(1+\m)/\d});
\draw [dashed,opacity=.3] (0,0) -- (1,1);
\draw [dashed,opacity=.1]
(.5, 0) -- (.5, .5);
}

\newcommand{\drawQbg}{
\fill (Q1) node [left] {$Q_1$} circle [radius=.02em];
\fill (Q2) node [above] {$Q_{2,\beta}$} circle [radius=\ptsize];
\fill (Q3) node [right] {$Q_{3,\beta}$} circle [radius=\ptsize];
\fill (Q4g) node [below] {$Q_{4,\gamma}$} circle [radius=\ptsize];
\fill [color=\QbgFillColor,opacity=\QbgFillOpacity] (Q1) -- (Q2) -- (Q3) -- (Q4g) -- cycle;
\draw [opacity=\QbgLineOpacity] (Q1) -- (Q2) -- (Q3);
\draw [opacity=\QbgLineOpacity] (Q4g) -- (Q1);
\if\QbgDrawCritSeg1
\draw [style=\QbgCritSegStyle,opacity=\QbgCritSegOpacity] (Q3) -- (Q4g);
\fi
}

\newcommand{\ttf} {\tfrac{x'-y'}{t}}
\newcommand{\ttfw} {\tfrac{x'-w'}{t}}

\def\lc{\lesssim}
\def\gc{\gtrsim}

\def\eps{\varepsilon}

\newcommand{\floor}[1]{\lfloor #1 \rfloor }

\newcommand{\Be}{\begin{equation}}
\newcommand{\Ee}{\end{equation}}

\newcommand{\Bm}{\begin{multline}}
\newcommand{\Em}{\end{multline}}

\def\intslash{\rlap{\kern .32em $\mspace {.5mu}\backslash$ }\int}
\def\qsl{{\rlap{\kern .32em $\mspace {.5mu}\backslash$ }\int_{Q_x}}}

\def\lc{\lesssim}
\def\gc{\gtrsim}

\def\floor#1{{\lfloor #1 \rfloor }}
\def\emph#1{{\it #1 }}
\def\diam{{\text{\rm diam}}}

\def\ga{\gamma}

\def\cf{{\it cf}}

\def\rank{{\mathrm{rank}}}

\def\meas{{\mathrm{meas}}}

\newcommand{\rmqA}{{\mathrm {qA}}}
\newcommand{\rmA}{{\mathrm A}}
\newcommand{\rmM}{{\mathrm M}}

\def\inn#1#2{\langle#1,#2\rangle}

\def\ga{\gamma} 
\def\eps{\varepsilon}
\def\ep{\epsilon}

\def\la{\lambda} \def\La{\Lambda}

\def\om{\omega}

\def\fM{{\mathfrak {M}}}

\def\fs{{\mathfrak {s}}}

\def\fx{{\mathfrak {x}}}

\def\bbH{{\mathbb {H}}}

\def\bbN{{\mathbb {N}}}

\def\bbR{{\mathbb {R}}}

\def\bbZ{{\mathbb {Z}}}

\def\cA{{\mathcal {A}}}

\def\cI{{\mathcal {I}}}
\def\cJ{{\mathcal {J}}}
\def\cK{{\mathcal {K}}}

\def\cM{{\mathcal {M}}}

\def\cP{{\mathcal {P}}}

\def\cR{{\mathcal {R}}}

\def\cT{{\mathcal {T}}}

\def\cZ{{\mathcal {Z}}}

\def\emph#1{{\it #1}}
\def\textbf#1{{\bf #1}}

\def\beq{\begin{equation}}
\def\endeq{\end{equation}}

\def\bs{\begin{split}}
\def\es{\end{split}}

\theoremstyle{plain}

\newtheorem{thm}{Theorem}[section]
\newtheorem{prop}[thm]{Proposition}

\newtheorem{lem}[thm]{Lemma}

\newtheorem*{thm*}{Theorem}
\newtheorem*{conj*}{Conjecture}
\newtheorem*{openproblem*}{Open Problem}

\theoremstyle{remark}
\newtheorem{rem}[thm]{Remark}

\newtheorem*{remarka}{Remark}

\numberwithin{equation}{section}

\definecolor{recol}{rgb}{0,0,1.}
\definecolor{rscol}{rgb}{0.1,.6,0.1}
\definecolor{ascol}{rgb}{1.,0,0}

\def\R{\mathbb{R}}
\def\C{\mathbb{C}}

\begin{document}
\title
[Spherical maximal functions on Heisenberg groups]{
Spherical maximal operators on Heisenberg groups: Restricted dilation sets}
\author[ J. Roos \ \ A. Seeger \ \ R. Srivastava]{Joris Roos \ \ \ \ Andreas Seeger \ \ \ \ Rajula Srivastava}

\address{Joris Roos: Department of Mathematical Sciences, University of Massachusetts Lowell \& School of Mathematics, The University of Edinburgh}
\email{joris\_roos@uml.edu}

\address{Andreas Seeger: Department of Mathematics, University of Wisconsin, 480 Lincoln Drive, Madison, WI, 53706, USA.}
\email{seeger@math.wisc.edu}

\address{Rajula Srivastava: Department of Mathematics, University of Wisconsin, 480 Lincoln Drive, Madison, WI, 53706, USA.}
\email{rajulas@math.uni-bonn.de}
\curraddr{Mathematical Institute, University of Bonn, Endenicher Allee 60, 53115 Bonn, Germany}


\maketitle

\begin{abstract}
Consider spherical means on the Heisenberg group with a codimension two incidence relation,
and associated spherical local maximal functions $M_E f$ where the dilations are restricted to a set $E$. We prove $L^p\to L^q$ estimates for these maximal operators; the results depend on
various notions of dimension of $E$.
\end{abstract}

\section{Introduction}

The purpose of this paper is to extend recent $L^p$-improving results for local spherical maximal functions on the Heisenberg group in \cite{RoosSeegerSrivastava} to the setting of restricted dilation sets. To fix notation, for $n\in \bbN $, we let $\bbH^n$ denote the Heisenberg group of Euclidean dimension $d=2n+1$.
We denote coordinates on $\bbH^n$
by $x=(\ux,\ox)\in \bbR^{2n}\times \bbR$. The group law is given by
\[x \cdot y =(\ubar x+\ubar y, \ox+\oy + \ux^\intercal J\uy), \]
where $J$ is an invertible skew symmetric $2n\times 2n$ matrix.
The Heisenberg group is equipped with automorphic dilations given by
$\delta_t(x)=(t\ubar x, t^2 \ox)$.

Let $\mu$ be the normalized rotation-invariant measure on the $2n-1$ dimensional unit sphere in the horizontal subspace $\bbR^{2n} \times \{0\}$, centered at the origin.
The automorphic dilations map this subspace into itself. We define the dilates of $\mu$ by $\inn{\mu_t}{f}= \inn{\mu}{f\circ\delta_t}$, where $t>0$.
In this paper we study the averaging operators
\[f*\mu_t(x)=\int_{S^{2n-1}} f(\ubar x-t\om, \ox-t
\ux^\intercal J\om ) d\mu (\om),
\]
which were introduced
by Nevo and Thangavelu \cite{NevoThangavelu1997}.

Let $E\subset [1,2]$. We are interested in determining the set of exponent pairs $(\frac1p,\frac1q)\in [0,1]^2$
so that the local maximal operator
\[M_E f=\sup_{t\in E} |f*\mu_t|\]
extends to a bounded operator $L^p(\bbH^n)\to L^q(\bbH^n)$.
For the full maximal function $\sup_{t>0}|f*\mu_t|$
sharp $L^p(\bbH^n)\to L^p(\bbH^n)$ bounds for $n\ge 2$
were established
by M\"uller and the second author \cite{MuellerSeeger2004} and independently by Narayanan and Thangavelu \cite{NarayananThangavelu2004}.
The problem of $L^p\to L^q$ boundedness of the local version $M_{[1,2]}$ was investigated by Bagchi, Hait, Roncal and Thangavelu \cite{BagchiHaitRoncalThangavelu}, who were motivated
by applications to sparse bounds and weighted estimates for the corresponding {\em global} maximal function, as well as for a lacunary variant.
$L^p\to L^q$ results that are sharp up to endpoints, for both the single averages and full local maximal function,
were proved in our previous paper \cite{RoosSeegerSrivastava}.

In the present paper we ask what happens if we take for $E$ more general subsets of $[1,2]$.
This question was recently considered in the Euclidean setting in \cite{AndersonHughesRoosSeeger}, \cite{RoosSeeger} (also see the earlier paper \cite{SeegerWaingerWright1995} for the case $p=q$). While the $L^p\to L^p$ results depend on the Minkowski dimension of $E$ the new feature of \cite{AndersonHughesRoosSeeger}, \cite{RoosSeeger}
is the dependence on various different notions of fractal dimension.
These dimensions play a congruent role in the Heisenberg case. For $E\subset \bbR$ let $N(E,\delta)$ be the minimal number of intervals of length $\delta$ needed to cover $E$.
To state our main result we first recall the Minkowski and quasi-Assouad dimensions.
We say that $E$ has {\em Minkowski dimension} $\mathrm{dim}_M\,E=\beta\in [0,1]$ if for every $\varepsilon>0$ there exists
$c_\eps>0$ such that for every $\delta>0$,
\Be \label{eq:Minkentropy} N(E,\delta)\le c_\eps \delta^{-\beta-\varepsilon}.\Ee
The {\em Assouad spectrum} is a continuum of fractal dimensions defined in \cite{fraser-hare-hare-troscheit-yu} (see also \cite{fraser-yu1,fraser-yu2, fraser-book}): for $\theta\in [0,1]$ let $\overline{\dim}_{\mathrm{A},\theta}\, E$ denote the smallest number $\gamma$ such that for every $\eps>0$ there exists $c_\eps>0$ such that for every interval $I$ with $|I|\ge \delta^\theta$ we have
\Be \label{eq:Assouadentropy} N(E\cap I,\delta)\le c_\eps\, (|I|/\delta)^{\gamma+\varepsilon}.\Ee
As $\theta\mapsto \overline{\dim}_{\mathrm{A},\theta}\, E$ is non-decreasing the limit
$\dim_{\mathrm{qA}} \!E:= \lim_{\theta \nearrow 1} \overline{\dim}_{\mathrm{A},\theta}\, E $ exists and is called the \emph{quasi-Assouad dimension},
see \cite{lu-xi}.

To identify classes of sets for which our $L^p$-improving results are sharp we shall need the concept of quasi-Assouad regularity in \cite{RoosSeeger} (see also \cite{AndersonHughesRoosSeeger} for a related notion). A set $E\subset[1,2]$ with $\beta=\mathrm{dim}_{\mathrm{M}}\,E$ and
$\gamma~=~\mathrm{dim}_{\mathrm{qA}}\,E$ is called {\em quasi-Assouad regular} if either $\gamma=0$ or
$\overline{\dim}_{\mathrm{A},\theta}\,E=\dim_{\mathrm{qA}}\,E$ for all $\theta\in (1-\beta/\gamma,1)$. Observe that always $0\le \beta\le \gamma\le 1$.

Let $\mathcal{R}(\beta,\gamma)$ denote the closed quadrilateral with corners
\Be\label{quadrilateral}\begin{gathered}
Q_1=(0,0), \qquad Q_{2,\beta}=(\tfrac{2n-1}{2n-1+\beta},\tfrac{2n-1}{2n-1+\beta}),
\\
Q_{3,\beta} = (\tfrac{2n+1-\beta}{2n+3-\beta}, \tfrac{2}{2n+3-\beta}), \quad
Q_{4,\gamma} = (\tfrac{n(2n+1)}{2n^2+3n+2\gamma},
\tfrac{ 2n}{2n^2+3n+2\gamma}).
\end{gathered}
\Ee
\begin{figure}[ht]
\begin{tikzpicture}[scale=5]
\def\m{1}
\def\n{2}
\def\b{.8}
\def\g{.9}
\definecoords
\drawauxlines
\drawQbg
\end{tikzpicture}
\caption{The quadrilateral $\mathcal{R}(\beta,\gamma)$.}
\end{figure}
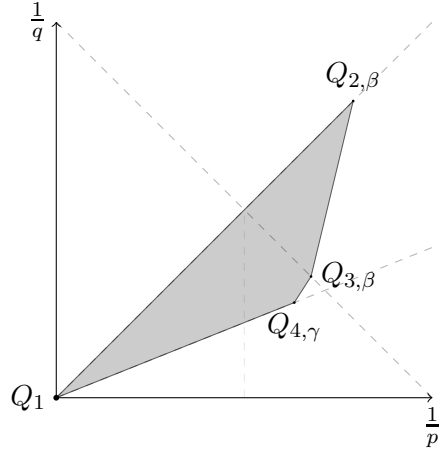

\begin{thm} \label{thm:main}
Let $n\ge 2$, $E\subset [1,2]$ with $\mathrm{dim}_\mathrm{M}\,E=\beta$ and $\mathrm{dim}_\mathrm{qA}\,E=\gamma$.
Then the following hold.

(i) $M_E:L^p(\bbH^n)\to L^q(\bbH^n)$ is bounded for $(\frac 1p, \frac 1q)$ in the interior of $\mathcal{R}(\beta,\gamma)$, and on the line segment $[Q_1,Q_{2,\beta})$.

(ii) If $E$ is quasi-Assouad regular and
$(\frac1p,\frac1q)\not\in \mathcal{R}(\beta,\gamma)$, then $M_E$
does not map $L^p(\bbH^n) $ to $L^q(\bbH^n)$.
\end{thm}

Note that up to endpoints we recover the corresponding sharp results for $E=[1,2]$ in \cite{RoosSeegerSrivastava}.
Further examples of quasi-Assouad regular sets include convex sequences, self-similar sets with $\beta=\gamma$ (such as Cantor sets)
and many more; see \cite[\S 6]{RoosSeeger}. Note that we do not cover the case $n=1$; indeed it is currently unknown whether the full circular maximal operator on the Heisenberg group $\bbH^1$ is bounded on any $L^p$ for $p<\infty$ and $L^p$-improving estimates are even more elusive (see \cite{BeltranGuoHickmanSeeger, LeeLee} for results on Heisenberg-radial functions).

The definitions of Minkowski and Assouad dimension in \eqref{eq:Minkentropy} and \eqref{eq:Assouadentropy} allow positive or negative powers of $\log\delta^{-1} $, or $\log (\delta/|I|)^{-1}$, and are therefore not suitable for the formulation of endpoint results at the boundary of $\cR(\beta,\gamma)$.
The following theorem covers such endpoint results for $0<\beta<1$. We define functions $\chi^E_{\rmM,\beta}, \,\chi^E_{\rmA,\gamma}:[0,1]\to [0,\infty)$, by
\begin{subequations}
\begin{align}
\label{eq:Minkchar}
\chi^E_{\rmM,\beta}(\delta)&= \delta^\beta N(E,\delta),
\\
\label{eq:Assouadchar}
\chi^E_{\rmA,\gamma}(\delta) &= \sup_{|I|>\delta} (\delta/|I|)^\gamma N(E\cap I, \delta)\,.
\end{align}
\end{subequations}
As in \cite{RoosSeeger} we refer to $\chi_{\rmM,\beta}^E$
as
the $\beta$ -Minkowski characteristic of $E$ and to
$\chi^E_{\rmA,\gamma}$
as the $\gamma$-Assouad characteristic of $E$.
If these characteristics are bounded then we obtain $L^p\to L^q$ boundedness of $M_E$ on the edges of $\cR(\beta,\gamma)$, with the possible exception of corners $Q_{2,\beta}$, $Q_{3,\beta}$ and $Q_{4,\gamma} $.
\begin{thm}
\label{thm:mainsharp}
Let $n\ge 2$, $E\subset [1,2]$, $0\le \beta\le 1$ and $\beta\le\gamma\le 1$ and assume that $\sup_{0<\delta<1} \chi^E_{\rmM,\beta}(\delta) <\infty$, $
\sup_{0<\delta<1} \chi^E_{\rmA,\gamma}(\delta) <\infty.$
Then the following hold.

(i)
$M_E :L^p(\bbH^n)\to L^q(\bbH^n)$ for $(\frac 1p, \frac 1q)\in \mathcal{R}(\beta,\gamma) \setminus \{ Q_{2,\beta}, Q_{3,\beta}, Q_{4,\gamma} \}$.

(ii) $M_E$ is of restricted weak type $(p,q)$ for all $(\frac 1p, \frac 1q) \in \cR(\beta,\gamma) $.
\end{thm}
The case $\beta=0$ corresponds to single averages for which a stronger result was proved in \cite{RoosSeegerSrivastava}.
The main ideas for the $L^p$ improving results in Theorems \ref{thm:main} and \ref{thm:mainsharp} follow roughly
the outline in the Euclidean case \cite{SchlagSogge1997, Lee2003, AndersonHughesRoosSeeger, RoosSeeger} {(even though the outcomes are quite different)} and there are also similarities to the treatment of the full maximal operators on Heisenberg groups $\bbH^n$ ($n\ge 2$) in \cite{RoosSeegerSrivastava}. However there is an important difference which makes the proof of the estimate at $Q_{4,\gamma}$ harder. Concretely, in the case of a restricted dilation set we can no longer efficiently use the space-time rotational curvature properties for the
averages
$(x,t)\mapsto f*\mu_t(x)$ which we relied on in \cite{RoosSeegerSrivastava}.
Unlike in the Euclidean case the fixed time averages $f*\mu_t$ do not have nonvanishing rotational curvature but are Fourier integral operators whose canonical relations project with fold singularities. As noticed in \cite{MuellerSeeger2004} this does not severely impact the outcome for the $L^p\to L^p$-inequalities for the maximal functions, however it creates technical problems in the proofs of the sharp $L^p$-improving estimates for $(1/p,1/q)$ away from the diagonals (\cf. \S\ref{sec:Q4}).

\subsection*{Further remarks and results} It is natural to ask what happens if in the above results one drops the assumpton that $E$ be quasi-Assouad regular.
There are many interesting examples, in particular unions
of quasi-Assouad regular sets are typically not quasi-Assouad regular. In the case of finite unions,
one can deduce from the above results that the closure of the sharp region of boundedness exponents is given by a polygon
arising as the intersection of finitely many quadrilaterals of the form $\mathcal{R}(\beta,\gamma)$.
When considering countable unions, more complicated convex regions can arise.
The following result is a direct analogue of a corresponding result in the Euclidean setting.

\begin{thm} \label{thm:generalE}Let $n\ge 2$ and let $\cT_E$ be the type set of $M_E$, i.e. the set of $(\tfrac 1p, \tfrac 1q)$ such that $M_E:L^p(\bbH^n)\to L^q(\bbH^n)$ is bounded. Then the following hold.

(i) Suppose that $E= \cup_{i=1}^N E_i$ where $E_i$ are quasi-Assouad regular sets with $\dim_{\rmM} E_i=\beta_i$ and $\dim_{\rmqA} E_i=\gamma_i$. Then $\overline {\cT_E}= \cap_{i=1}^N\cR(\beta_i,\gamma_i)$.

(ii) If $\dim_{\rmM}E=\beta$, $\dim_{\rmqA} E=\gamma$, then $\cR(\beta,\gamma) \subset \overline{\cT_E} \subset \cR(\beta,\beta) $.

(iii) For every closed convex set $\cT$ satisfying $\cR(\beta,\gamma)\subset \cT \subset \cR(\beta,\beta)$ there is a set $E\subset [1,2]$ with $\dim_{\rmM} E=\beta$ and $\dim_{\rmqA}E=\gamma$ such that $\overline{\cT_E}=\cT$.
\end{thm}

In particular, (ii) and (iii) characterize exactly which closed convex sets can arise as $\overline{\cT_E}$ for some $E\subset [1,2]$. It turns out that the essential sharpness of the results for quasi-Assouad regular dilation sets in Theorem \ref{thm:main} allows one to give a proof of Theorem
\ref{thm:generalE} that is entirely analogous to the arguments
in \cite[\S 5-7]{RoosSeeger} and we will therefore not repeat the details of the constructions here.

Our results have applications to sparse bounds for global maximal operators
given by $f\mapsto \sup_{k\in \bbZ}\sup_{t\in E}|f*\mu_{2^k t}| $.
We refer to the detailed discussion in the paper by Bagchi, Hait, Roncal, Thangavelu \cite{BagchiHaitRoncalThangavelu} who show how (partial) results on $L^p$-improving esimates imply corresponding partial results on sparse bounds for the lacunary and full maximal functions
(see also \cite[\S 8]{RoosSeegerSrivastava} for a discussion of an essentially sharp version of such results). In the same way our results imply sparse bounds for the global maximal operators with restricted dilation sets.

Finally we remark that the behavior of maximal operators associated with the codimension two spherical means considered here is quite different from the behavior of maximal functions associated with {\it hypersurfaces} in the Heisenberg group. Of particular interest here is the Kor\'anyi sphere, for which the sharp $L^p$ improving properties of the local full maximal operator up to endpoints were obtained in a recent paper by one of the authors \cite{Srivastava}, see also partial results about averages in previous work \cite{GangulyThangavelu} by Ganguly and Thangavelu.

\subsection*{Summary of the paper}
\begin{itemize}[topsep=5pt]
\item[--] \S \ref{sec:prelim} contains some known preliminary reductions.
\item[--] \S\ref{sec:basicbounds} contains the proof of the basic bounds at the points $Q_1, Q_{2,\beta}, Q_{3,\beta}$ and states the estimates proving part (i) of Theorems \ref{thm:main}, \ref{thm:mainsharp} (i).
\item[--] \S \ref{sec:Q4} is concerned with the estimate at $Q_{4,\gamma}$. We follow the main argument in \S \ref{sec:Q4} which is the reduction to an $L^2\to L^q$ estimate. This is handled by $TT^*$ arguments similar to \cite{AndersonHughesRoosSeeger}, but we have to overcome difficulties caused by the presence of fold singularities. These arguments complete the proof of part (i) of Theorems \ref{thm:main}, \ref{thm:mainsharp}.
\item[--] In \S \ref{sec: kernelest} we prove the key kernel estimate, Proposition \ref{prop:kernelest}.
\item[--] In \S \ref{sec:necessary} we prove part (ii) of Theorems \ref{thm:main}, \ref{thm:mainsharp} by testing the operator on some old and new counterexamples.
\end{itemize}

\subsection*{Notation}
Partial derivatives will often be denoted by subscript.
$P$ denotes the $(2n-1)\times2n$ matrix $P=(I_{2n-1}\,\, 0)$. \detail{Then $PJ$ is a $(2n-1)\times2n$ matrix with its rows identical to the first $2n-1$ rows of $J$.}
By $A\lesssim B$ we mean that $A\le C\cdot B,$ where $C$ is a constant and $A\approx B$ signifies that $A\lesssim B$ and $B\lesssim A$.

\subsection*{Acknowledgements}
\thanks{The authors would like to thank
the Hausdorff Research Institute of Mathematics and the organizers of the trimester program {\it Harmonic Analysis and Analytic Number Theory} for a pleasant working environment during their visits in the summer of 2021. This work was supported by National Science Foundation grants DMS-2054220 (A.S.) and DMS-2154835 (J.R.), and by a grant from the Simons Foundation (ID 855692, J.R.).
The authors also thank a referee for a thorough reading and suggestions that have improved the exposition.
}

\section{Preliminaries}\label{sec:prelim}
Via suitable rotation and localization arguments (as explained in Section 2.1 of \cite{RoosSeegerSrivastava}), we may assume that $f$ is supported in a small neighborhood of the origin and the measure $\mu$ is supported in a small neighborhood of the vector $e_{2n}$.
Splitting $\ubar y=(y', y_{2n})$ and using the parametrization $\om=(w', g(w')) $ with $g(w')=\sqrt{1-|w'|^2}$ near the north pole $e_{2n}$ of the sphere,
it suffices to consider maximal functions $\sup_{t\in E}|\cR f(x,t)|$ where the integral operator (generalized Radon transform) $\cR $ is defined by
\[\cR f(\ux,\ox,t) =\int \chi(x,t,y') f(y', \fs^{2n}(x,t,y'), \overline \fs(x,t,y')) dy';
\] here $\chi$ is smooth and supported on
\Be\label{apriori-support-assumptions} \{(x',x_{2n}, \ox ,t,y' ): |y'|\le\epsilon,\, |x'|\le\epsilon,\, |x_{2n}-t|\le \epsilon, \, |\ox |\le\epsilon\}.
\Ee
The choice of $\ep$ will be determined by considerations in the proof of Lemma \ref{lem:oio} below, depending on the size of derivatives of phase functions and the choice of $J$, but it is not necessary to track this.
\begin{subequations}
\begin{align} \label{definition of fs}
\fs^{2n}(x,t,y') &=x_{2n}- tg (\tfrac{x'-y'}t) , \\
\bar \fs(x,t,y') &=
\ox
+ \ux^\intercal JP^\intercal y' +\big(x_{2n}-tg(\tfrac{x'-y'}{t})\big)
(\ux^\intercal J e_{2n}
),
\end{align}
\end{subequations}
where $P= \begin{pmatrix}I_{2n-1} &0\end{pmatrix}$ is the matrix of the projection on $\bbR^{2n}$ omitting the last coordinate.
We will need that
\begin{equation}\label{eq:Psidefc}
g(0)=1,\; \nabla g(0)=0,\; g''(0)= -I_{2n-1},\; g'''(0)=0.
\end{equation}
It will be convenient to introduce a nonlinear shear transformation in the $x$-variables
\begin{align*}
\underline \fx(x)&=\underline x,
\\
\overline \fx(x)&=\ox -x_{2n}\,{\underline x}^\intercal Je_{2n}.
\end{align*}
By a change of variables it suffices to prove the relevant estimates for
$\cA f(x,t)= \cR f(\fx(x),t).$
The operator $\cA$ has a Schwartz kernel which is a co-normal distribution given by
\[K(x,t,y)= \chi_1 (x,t,y') \delta_0( S^{2n}(x,t,y') -y_{2n} , \overline S (x,t, y')-\overline y),\] where $\chi_1(x,t,y')=\chi(\fx(x),t,y')$, $\delta_0$ is Dirac measure at the origin in $\bbR^{2}$ and
$(S^{2n}, \bar S)|_{(x,t,y')}= (\fs^{2n}, \bar\fs)|_{ (\fx(x),t,y')}$, that is
\Be\label{eq:PhaseS}\begin{aligned}
S^{2n}(x,t,y')&= x_{2n}- tg(\ttf),
\\
\bar S(x,t,y')&=\ox +
(\ux^\intercal J) (P^\intercal y' -tg(\ttf)
e_{2n}),
\end{aligned}\Ee
with $g$ as in \eqref{eq:Psidefc}.
Note that the function $\chi_1$ is still supported in a set of the form \eqref{apriori-support-assumptions}, where we replace $\ep$ by $O(\ep)$.
It is standard to express $\delta_0$ via the Fourier transform \Be\label{eq:delta=expansion}K(x,t,y) =
\chi_1 (x,t,y') \int_{\theta\in \bbR^{2}}
e^{i\Psi(x,t,y,\theta)} \tfrac{d\theta}{(2\pi)^{2}}, \Ee
where the phase function $\Psi$ is given by \begin{equation}
\label{phasedefn}
\Psi(x,t,y,\theta)=\theta_{2n}(S^{2n}(x,t,y')-y_{2n}) +\bar\theta (\bar S (x,t,y') -\bar y)
\end{equation}
and $\theta=(\theta_{2n}, \bar\theta)$.

We now perform a dyadic decomposition of this modified kernel.
Let $\zeta_0$ be a smooth radial function on $\R^{2}$ with compact support in $\{|\theta|< 1\}$ such that $\zeta_0(\theta)=1$ for $|\theta|\le 1/2$. We set $\zeta_1(\theta)=\zeta_0(\theta/2)-\zeta_0(\theta)$ and $\zeta_j(\theta)= \zeta_1(2^{1-j}\theta)$ for $j\ge 1$.

We set, for $k=0,1,2,\dots$
\begin{equation} \notag
\cA^{k}_t f(x) = \int
\chi_1(x,t,y')\int_{\theta\in \bbR^{2}}\zeta_k(\theta)
e^{i\Psi(x,t,y,\theta)} \tfrac{d\theta}{(2\pi)^{2}}\, f(y) dy.
\end{equation} For $k\ge 1$ this can be rewritten, by a change of variables and the homogeneity of the phase function with respect to $\theta$, as
\begin{equation}
\label{Littlewood-Paley}
\cA^{k}_t f(x) = 2^{2k}\int
\chi_1(x,t,y')\int_{\theta\in \bbR^{2}}\zeta_1(2\theta)
e^{i2^k \Psi(x,t,y,\theta)} \tfrac{d\theta}{(2\pi)^{2}}\, f(y) dy.
\end{equation}

As already observed in \cite{MuellerSeeger2004} these Fourier integral operators lack ``rotational curvature" (i.e. the assumption that the associated canonical relation is locally the graph of a diffeomorphism). Indeed from H\"ormander \cite{hormanderFIO} the ``rotational curvature matrix'' is given by
\[\mathrm{Rot curv}(\Psi) =\begin{pmatrix} \Psi_{xy}&\Psi_{x\theta}\\ \Psi_{\theta y}&\Psi_{\theta\theta} \end{pmatrix}
\] which is equal to
\[
\begin{pmatrix}
\theta_{2n} S^{2n}_{x'y'} + {\bar\theta}\, \bar S_{x'y'} & 0&0&S^{2n}_{x'} & \bar S_{x'}
\\
{\bar\theta}\,e_{2n}^\intercal JP^\intercal &0&0&1&0\\
0&0&0& 0&1\\
(S^{2n}_{y'})^\intercal &-1&0&0&0
\\
(\bar S_{y'})^\intercal &0&-1&0&0
\end{pmatrix}. \]
One calculates
$S^{2n}_{x'}=-\nabla g(\ttf)$, $S^{2n}_{y'}=\nabla g(\ttf)$ and \begin{multline*}
\theta_{2n} S^{2n}_{x'y'} + {\bar\theta}\, \bar S_{x'y'}=\\
t^{-1}(\theta_{2n}+{\bar\theta}\, \ux^\intercal Je_{2n})g''(\ttf)+{\bar\theta}\,\left[PJP^\intercal+B(x,t,y')\right]
\end{multline*}
where the $(2n-1)\times (2n-1)$ matrix $B(x,t,y')$ is given by
\[B(x,t,y')=PJe_{2n}\nabla g(\ttf).\]

With \Be\label{sigmadefinition} \sigma (x,\theta)= \theta_{2n}+{\bar\theta} \,\ux^\intercal Je_{2n}\Ee we see that $\mathrm{Rot curv}(\Psi)$ equals
\[
\begin{pmatrix}
t^{-1}\sigma g''(\ttf)+{\bar\theta}(PJP^\intercal+B)&0&0&-\nabla g(\ttf)&*\\
{\bar\theta}\,e_{2n}^\intercal JP^\intercal &0&0&1&0\\
0&0&0& 0&1\\
\nabla g(\ttf)^\intercal &-1&0&0&0
\\
* &0&-1&0&0
\end{pmatrix}\] and by using elementary row operations and the skew-symmetry of $J$, it is not hard to see that \[\det{(\mathrm{Rot curv}(\Psi))}=\det{\left(t^{-1}\sigma g''(\ttf)+{\bar\theta}\,(PJP^\intercal+B-B^\intercal)\right)}.\] Note that $PJP^\intercal+B-B^\intercal$ is a skew-symmetric matrix of order $2n-1$ and is thus not invertible. Using \cite[Lemma 3.1] {RoosSeegerSrivastava}, we conclude that
\[t^{-1}\sigma g''(\ttf)+{\bar\theta}\,(PJP^\intercal+B-B^\intercal)\]
is invertible if and only if $\sigma\neq 0$.
Indeed from the calculations in \S3 of \cite{RoosSeegerSrivastava} and \cite[Lemma 5.4]{MuellerSeeger2004} it follows that
\[\det(\mathrm{Rotcurv}(\Psi))\approx \sigma (x,\theta).\]
It is natural to use an idea in \cite{PhongStein1991} to further decompose in terms of the size of $\sigma$
(see also \cite{Cuccagna1997}, \cite{MuellerSeeger2004}, \cite{BeltranGuoHickmanSeeger}).
For $k\ge 1$ and
$1\leq \ell\leq \lfloor\tfrac{k}{3}\rfloor$, we define
\Be\label{uell}
u_\ell(x,\theta) =
\begin{cases}
(1-\zeta_0( \tfrac 12\sigma (x,\theta) )) &\text{ if } \ell=0,
\\
\zeta_1(2^\ell \sigma(x,\theta) ) &\text{ if } 1\le \ell< \floor{k/3},
\\
\zeta_0( 2^{\floor{ k/3}} \sigma(x,\theta)) &\text{ if } \ell= \floor{k/3},
\end{cases}
\Ee so that $\sum_{\ell=0}^{\floor{\frac k3}} u_\ell=1$ and $u_\ell$ is supported where $|\sigma|\approx 2^{-\ell} $ when $1\le \ell< \floor{k/3}$. Set
\begin{multline} \label{Akell}
\cA^{k,\ell}_t f(x) \equiv \cA^{k,\ell} f(x,t)\\=2^{2k}\int
\chi_1(x,t,y')\int_{\theta\in \bbR^{2}} \zeta_1(2\theta) u_\ell(x,\theta)
e^{i2^k\Psi(x,t,y,\theta)} \tfrac{d\theta}{(2\pi)^{2}}\, f(y) dy.
\end{multline}
Furthermore, for $k\geq 1$ and $0\leq \ell\leq \lfloor\tfrac{k}{3}\rfloor$, we let $\cM^0_E f(x)= \sup_{t\in E} |\cA^0_t f(x)|,$

\Be\label{cM-definitions}
\cM^{k,\ell}_E f(x)= \sup_{t\in E} |\cA^{k,\ell}_t f(x)|\text{ and }\,\,
\cM^k_E f(x)= \sum_{0\leq\ell\leq \lfloor\tfrac{k}{3}\rfloor}\cM^{k,\ell}_E f(x).\Ee
Since for all $E\subset [1,2]$ the operator $\cM_E^0$ maps $L^p\to L^q$ for all $1\le p\le q\le \infty$ it will be ignored in what follows.

\subsection{\texorpdfstring{\it The operators $\partial_t \cA^{k,\ell}$ versus $\cA^{k,\ell}$}{The operators versus their derivatives}} \label{sec:opvsderiv}

Finally, in order to estimate the maximal operators $\cM^{k,\ell}_E$, we will rely on estimates for $\partial_t \cA^{k,\ell} f(x,t)$.
As in \cite{MuellerSeeger2004}, \cite{BeltranGuoHickmanSeeger} it will be crucial to observe that $\partial_t \Psi$ lies in the ideal generated by $\sigma$, indeed
\Be\label{eq:Psi-tderiv}
\partial_t\Psi(x,y,\theta) = - \partial_t \big( t g(\tfrac{x'-y'}t)\big) \sigma(x,\theta).
\Ee
In view of \eqref{uell}, \eqref{Akell}, \eqref{eq:Psi-tderiv} the operator $2^{\ell-k} \partial_t \cA^{k,\ell} $ will usually have the same quantitative behavior as $\cA^{k,\ell} $.

To expand on this let $\cK^{k,\ell}(x,t) $ be the Schwartz kernel of
$\cA^{k,\ell}$, i.e.
\Be\label{eq:K_k}
\cK^{k,\ell}(x,t,y) =
2^{2k}\int_{\bbR^{2}}e^{i2^k\Psi(x,t,y,\theta)} a_{\ell}(x,t,y',\theta) d\theta, \Ee
with $\Psi$ as in \eqref{phasedefn} and $a_\ell(x,t,y',\theta)=(2\pi)^{-2}\chi_1(x,t,y')\zeta_1(2\theta) u_\ell(x,\theta)$.

For the $t$-derivatives we compute
\begin{align*} \partial_t \cK^{k,\ell}(x,t,y) &= i 2^{3k} \int \partial_t\Psi(x,t,y,\theta) e^{i2^k \Psi(x,t,y,\theta)} a_\ell(x,t,y',\theta)\,d\theta \\& +\,2^{2k} \int e^{i2^k \Psi(x,t,y,\theta)} \partial_t a_\ell(x,t,y',\theta)\,d\theta. \end{align*}
Observe that $\partial_t a_\ell(x,t,y',\theta)=(\partial_t \chi_1(x,t,y')) \zeta_1(2\theta) u_\ell(x,\theta)$ and its derivatives satisfy the same quantitative estimates as $a_\ell$.
Regarding the first summand we use \eqref{eq:Psi-tderiv}.
The expression $\partial_t ( t g(\tfrac{x'-y'}t ))$ does not depend on $\theta$ and its derivatives satisfy uniform bounds.
Since $|\sigma(x,\theta)|\approx 2^{-\ell}$ we see that
the modified amplitude function
\[ \widetilde{a}_{\ell}(x,t,y',\theta) = 2^\ell \sigma(x,\theta) a_\ell(x,t,y',\theta)\]
satisfies the same estimates as $a_\ell$, with a similar statement for the derivatives.
As a consequence of these considerations we see that the operator
$2^{-k+\ell} \partial_t \mathcal{A}^{k,\ell}_t$ will always satisfy the same estimates as $\mathcal{A}^{k,\ell}$, and we usually omit a separate proof for $\partial_t\cA^{k,\ell}_t$.

\section{
Basic estimates}\label{sec:basicbounds}
We use the representation \eqref{eq:K_k} for the Schwartz kernel $\cK^{k,\ell} $ of $\cA^{k,\ell}$
and integration by parts yields the
estimate
\begin{multline} \label{eq:Kkptw}
|\cK^{k,\ell}(x,t,y)|\le C_N \frac{ 2^{k-\ell} } {(1+2^{k-\ell}|y_{2n} - S^{2n} ( x, t,y')|)^N}\\
\times \frac{ 2^{k} } {(1+2^k |\bar y-\bar S (x,t,y') - \ux^\intercal Je_{2n}
(y_{2n} - S^{2n} ( x, t,y')) |)^N} .
\end{multline}

This estimate
yields
\begin{align*} \sup_{t\in [1,2]} \sup_{x,y} |\cK^{k,\ell}(x,t, y)| &\lc 2^{2k-\ell},
\\
\sup_{t\in [1,2]} \sup_{x}\int |\cK^{k,\ell}(x,t, y)| dy &\lc 1,
\\
\sup_{t\in [1,2]} \sup_{y}\int |\cK^{k,\ell}(x,t, y)| dx &\lc 1,
\end{align*}
where for the third inequality we used the specific expressions for $S^{2n}$, $\bar S$ in \eqref{eq:PhaseS}. It follows that
\begin{gather}\label{eqn:trivestimate}
\|\cA_t^{k,\ell}\|_{L^1\to L^1} + \|\cA_t^{k,\ell} \|_{L^\infty\to L^\infty} \lesssim 1,
\\ \label{eqn:trivL1infestimate}
\|\cA_t^{k,\ell}\|_{L^1\to L^\infty} \lesssim 2^{2k-\ell}.
\end{gather}
We also have the $L^2$ fixed-time estimate
\begin{equation}\label{eqn:singleL2estimate}
\| \cA^{k,\ell}_t f\|_{L^2(\R^{2n+1})} \lesssim
2^{-k \frac{2n-1}{2}}2^{\frac{\ell}{2}}\|f\|_2,
\end{equation}
for $0\leq \ell\leq \tfrac{k}{3}$. Display \eqref{eqn:singleL2estimate} was established in \cite{MuellerSeeger2004} via estimates for oscillatory integrals with fold singularities in \cite{Cuccagna1997}, see also the detailed treatment of a relevant extended class of oscillatory integral operators in \cite[\S6]{BeltranGuoHickmanSeeger}.
By interpolation we get
\begin{prop} \label{prop:single-avg-est} Let $n\ge 1$ and $t\in [1,2]$.

(i) For $1\le p\le 2$,
\Be \label{eqn:singleLpest1}
\| {\cA}^{k,\ell}_t f\|_{p} \lesssim 2^{-k\frac{2n-1}{p'}}2^{\frac{\ell}{p'}} \|f\|_p
\Ee
and for $2\leq p\leq \infty$,
\Be \label{eqn:singleLpest2}
\| {\cA}^{k,\ell}_t f\|_{p} \lesssim 2^{-k\frac{2n-1}{p}}2^{\frac{\ell}{p}} \|f\|_p.
\Ee
(ii) For $2\le q\le \infty$,
\Be \label{eqn:singleantidiagest}
\|{\cA}^{k,\ell}_t\|_{q}\lesssim 2^{k (2-\frac{2n+3}{q})}2^{\ell (\frac{3}{q}-1)} \|f\|_{q'}.
\Ee
(iii) The same estimates hold for $2^{-k+\ell} \partial_t \cA^{k,\ell}_t$ in place of $\cA^{k,\ell}$.
\end{prop}
\begin{proof}
Part (i) follows by interpolating between \eqref{eqn:trivestimate} and \eqref{eqn:singleL2estimate}, while Part (ii) is a consequence of interpolating between \eqref{eqn:trivL1infestimate} and \eqref{eqn:singleL2estimate}.
For part (iii) see the considerations in \S\ref{sec:opvsderiv}.
\end{proof}
The above estimates give the following bounds for the maximal operator $\cM^{k,\ell}_E$.

\begin{prop} For all $n=1,2,3,\dots$ we have the following bounds for Schwartz functions $f$ on $\bbR^{2n+1}$.
\label{prop:Akmax}

(i) For $1\le p\le \infty$,
\Be\label{eqn:maxLpest}
\| \cM^{k,\ell}_E f\|_p \lc N(E, 2^{\ell-k})^{1/p}
2^{-k(2n-1)\min(\frac1p,\frac1{p'})}2^{\ell\min(\frac1p,\frac1{p'})}\|f\|_p.\Ee

(ii) For $2\le q\le \infty$,
\Be \label{eqn:maxantidiagest}
\|\cM^{k,\ell}_E f\|_{q} \lesssim N(E, 2^{\ell-k})^{1/q}
2^{k (2-\frac{2n+3}{q})} 2^{\ell (\frac{3}{q}-1)} \|f\|_{q'}.\Ee

(iii) If $\dim_{\rmM} E=\beta$, then for every $\eps>0$
\[ \| \cM^{k,\ell}_E f\|_p \lc_\eps
2^{(k-\ell)\frac{\beta+\eps}p}
2^{-k(2n-1)\min(\frac1p,\frac1{p'})}2^{\ell\min(\frac1p,\frac1{p'})}\|f\|_p, \quad 1\le p\le\infty
\]
and
\[ \|\cM^{k,\ell}_E f\|_{q} \lesssim_\eps 2^{(k-\ell)\frac{\beta+\eps}q}
2^{k (2-\frac{2n+3}{q})} 2^{\ell (\frac{3}{q}-1)} \|f\|_{q'}, \quad 2\le q\le\infty.
\]

\end{prop}
\begin{proof}
The fundamental theorem of calculus implies the pointwise bound
\[\cM^{k,\ell}_E f(x)\leq \sup_{t\in \cZ_{k-\ell}} \Big( |\cA^{k,\ell}_{t}f(x)|+\int_0^{2^{-k+\ell}}|\partial_s \cA^{k,\ell}_{t+s}f(x)|\,ds\Big),\]
where $\cZ_{k-\ell}$ consists of the left endpoints of a minimal collection of intervals of length $2^{-k+\ell}$ that covers $E$.
With this in hand, parts (i) and (ii) follow directly from Proposition \ref{prop:single-avg-est}.
Part (iii) is immediate since
\[N(E,2^{-k+\ell})\lesssim_\eps 2^{(k-\ell) (\beta+\eps)}\]
when $\dim_{\rmM} E=\beta$.
\end{proof}
For $\ell\geq 0$, we introduce the operator
\begin{equation}
\label{eqn:def ml op}
\fM_E^{\ell} :=\sum_{k\ge 3\ell} \cM_E^{k,\ell}.
\end{equation}
\begin{prop}
Let $n\ge 2$, $\beta\in (0,1]$ and assume that \[ \sup_{\delta>0} \chi^E_{\rmM,\beta}(\delta)\equiv \sup_{\delta>0}\delta^{\beta} N(E,\delta)\le A_1<\infty. \]
Let
$Q_{2,\beta}=(\tfrac{2n-1}{2n-1+\beta},\tfrac{2n-1}{2n-1+\beta})$ and
$ Q_{3,\beta} = (\tfrac{2n+1-\beta}{2n+3-\beta}, \tfrac{2}{2n+3-\beta})$.

(i) If $(1/p, 1/q)$ is one of the points
$Q_{2,\beta}$, $Q_{3,\beta}$ then there is $\alpha(p,q)>0$ such that
\[ \|\fM^\ell_E f\|_{L^{q,\infty} }\lc A_1^{1/q} 2^{-\ell\alpha(p,q)} \|f\|_{L^{p,1} }\]
and
$M_E: L^{p,1}\to L^{q,\infty}$ is bounded.

(ii) If $(1/p, 1/q)$ belongs to the open line segment connecting
$Q_{2,\beta}$ and $Q_{3,\beta}$ then
\[ \|M_E f\|_{L^{q,r}} \lc A_1^{1/q} \|f\|_{L^{p,r} } \]
for all $r>0$, in particular $M_E$ is bounded from $L^p$ to $L^q$.
\end{prop}

\begin{proof}
We observe that part (ii) follows from part (i) by real interpolation (note that the line connecting $Q_{2,\beta}$ and $Q_{3,\beta} $ has a positive finite slope).

We have, for $1\le p\le 2$,
\[
\| \cM^{k,\ell}_E f\|_p \lc 2^{-k( 2n-1-\frac{2n-1+\beta}p)} 2^{\ell(1-\frac{1+\beta}{p})}
\|f\|_p,\]
by Proposition \ref{prop:Akmax} (i).
By Bourgain's restricted weak type interpolation trick (see \cite{Bourgain85}, or the appendix of \cite{CarberySeegerWaingerWright1999}), applied to $\fM^{\ell}_E$ defined in \eqref{eqn:def ml op},
we get
\[
\Big\| \fM^\ell_E f\Big\|_{L^{p_{\mathrm cr},\infty}}
\lc
2^{\ell(1-\frac{1+\beta}{p_{\mathrm cr}})}
\|f\|_{L^{p_{\mathrm cr},1}}, \quad p_{\mathrm cr}= \tfrac{2n-1+\beta}{2n-1} \]
and we have $1-\frac{1+\beta}{p_{\mathrm{cr}}} =-\beta\frac{2n-2}{2n-1+\beta}$ so that we can sum in $\ell\ge 0$ if $n\ge 2$. The asserted restricted weak type inequality for $Q_{2,\beta}$ follows.

To prove the estimate for $Q_{3,\beta}$ we note that for $2\le q\le \infty$ we get
from Proposition \ref{prop:Akmax} (ii)
\[
\|\cM_E^{k,\ell} f\|_q\lc A_1^{1/q} 2^{k(2-\frac{2n+3-\beta}{q})} 2^{\ell (\frac {3-\beta}{q}-1)} \|f\|_{q'}
\]
and again by the restricted weak type interpolation result,
\[\|\fM_E^\ell f\|_{q}\lc 2^{\ell(\frac{3-\beta}{q_{\mathrm cr}} -1)}\|f\|_{q'}, \quad q_{cr}= \tfrac{2n+3-\beta}{2}.\]
We have
$\frac{3-\beta}{q_{\mathrm cr}} -1= \frac{3-\beta-2n}{3-\beta+2n} $ which is negative for $n\ge 2$. Summing in $\ell$ yields the desired result on $M_E$.
\end{proof}
Finally, we state the main estimate at the vertex
\Be\label{Q4display} Q_{4,\ga}=(\tfrac{1}{p_4}, \tfrac{1}{q_4})=(\tfrac{n(2n+1)}{2n^2+3n+2\gamma},\tfrac{2n}{2n^2+3n+2\gamma}).
\Ee
\begin{prop}
\label{prop: maxL2qest}
Let $n\geq 1$, $\gamma\in (0,1)$ and assume that
\[ \sup_{\delta>0}\chi^E_{\rmA,\gamma}(\delta) \equiv\sup_{\delta>0} \sup_{|I|>\delta} (\delta/|I|)^\gamma N(E\cap I, \delta)\leq A_2<\infty\,.\]
Let $p_4$, $q_4$ as in \eqref{Q4display}. Then
\begin{equation}
\label{eqn:Q4est}
\|\fM^\ell_E f\|_{q_4, \infty}\lesssim_{b}
A_2^{1/q_4} 2^{-\ell b}\|f\|_{p_4,1},\quad
\text{ for $b<\tfrac{n(2n-3)+2\gamma(n-1)}{2n^2+3n+2\gamma} $}.
\end{equation}
If in addition $n\ge 2$ then also \Be \label{nge2concl}\|M_E f\|_{L^{q_4,\infty}} \lc A_2^{1/q_4} \|f\|_{L^{p_4,1} }. \Ee

\end{prop}

The assertion for $M_E$ follows from \eqref{eqn:Q4est} after summing in $\ell$. Inequality \eqref{eqn:Q4est} will be proven in the next section as a consequence of Proposition \ref{prop:q5} below.

\section{\texorpdfstring{ Estimates at $Q_{4,\gamma}$}{The point Q4}}\label{sec:Q4}
After a decomposition of $E$ into a finite number of subsets we may assume that \Be \label{eq:t-diamter} \diam(E)<\ep,
\Ee
with $\ep$ as in \eqref{apriori-support-assumptions}.
Given a non-negative integer $m$, let $\cI_{m}(E)$ denote the set of all dyadic intervals of the form $(\nu2^{-m}, (\nu+1)2^{-m})$ (with $\nu\in \mathbb{Z}$) which intersect $E$. Then one observes that $\#\cI_{m}(E)\lesssim N(E,2^{-m})$. Thus, for any interval $I$ of length at least $2^{-m}$, we have
\[\# \cI_{m}(E\cap I)\lesssim \chi_{\mathrm{A},\gamma}^E(2^{-m})\, |I|^{\gamma}2^{m\gamma}.\]
Further, let $\cZ_m(E)$ denote the set of left endpoints of intervals $I_{\nu}\in \cI_{m}(E)$, endowed with the counting measure.
The main result of this section is the following Stein-Tomas type estimate for $\cA^{k,\ell}$.
\begin{prop}
\label{prop:q5}
Let $n\geq 1$ and $q_5=\frac{2(n+\gamma)}{n}$. Suppose
\[ \sup_{\delta>0}\chi^E_{\rmA,\gamma}(\delta)
\leq A_2<\infty\,,\]
where $\chi^E_{\rmA,\gamma}(\delta)$ is as defined in \eqref{eq:Assouadchar}.
Then for any $b_1>\frac{n(1-\gamma) }{2(n+\gamma) } $, we have
\Be \label{eqn:Q5single}
\| \cA^{k,\ell} \|_{L^2(\bbR^{2n+1})\to L^{q_5, \infty}(\R^{2n+1}\times \cZ_{k-\ell})} \lc_{b_1}
2^{-k(\frac{2n+1}{q_5} -1) }2^{\ell b_1 }.
\Ee
\end{prop}

\begin{proof}[Proof that Proposition \ref{prop:q5} implies Proposition \ref{prop: maxL2qest}]By the fundamental theorem of calculus,
\[\cM^{k,\ell}_Ef(x)\leq \sup_{t\in \cZ_{k-\ell}} \Big(|\cA^{k,\ell}_{t}f(x)|+\int_0^{2^{\ell-k}}|\partial_s\cA^{k,\ell}_{t+s}f(x)|\,ds\Big).\]
Thus, taking an $L^{q_5, \infty}$ norm on both sides and using Proposition \ref{prop:q5}, we conclude that
\begin{align} \notag
\|\cM^{k,\ell}_E f\|_{L^{q_5, \infty}}&\leq \| \cA^{k,\ell} f\|_{L^{q_5, \infty}(\R^{2n+1}\times \cZ_{k-\ell})}
\\ \notag
&\qquad\qquad +\int_0^{2^{\ell-k}}\|\partial_s\cA^{k,\ell}_{t+s}f(x)\|_{L^{q_5, \infty}(\R^{2n+1}\times \cZ_{k-\ell})}\,ds \\ \label{MEkell-alln}
&\lesssim 2^{-k(\frac{2n+1}{q_5} -1) }2^{\ell\frac{n}{2}\big(\frac{1-\gamma}{n+\gamma}\big) }2^{\ell\varepsilon}\|f\|_{2}.
\end{align}

We can now use Bourgain's trick to interpolate between the above estimate
and the case $q=\infty$ of
\eqref{eqn:maxantidiagest}, with $\vartheta=\tfrac{4(n+\gamma)}{2n^2+3n+2\gamma}\in (0,1)$, $a=\frac{n(2n-3)+2\gamma(n-1)}{2n^2+3n+2\gamma}$ and small $\eps>0$. Observe that $a>0$ if $n\ge 2$.
Note that
\[
(1-\vartheta)(1,0,2,-1)+\vartheta(\tfrac{1}{2},\tfrac1{q_5},1-\tfrac{2n+1}{q_5},\tfrac{n(1-\gamma)}{2(n+\gamma)}+\varepsilon)
=(\tfrac{1}{p_4},\tfrac{1}{q_4},0,-a+\vartheta\varepsilon)\]
which implies the desired estimate \eqref{eqn:Q4est}.
\end{proof}

\subsection*{Outline of the proof of Proposition \ref{prop:q5}}
We can use a partial scaled Fourier transform
\[F_k(y',\theta_{2n},\bar\theta)= \int_{\bbR^{2}} f(y',y_{2n}, \bar y) e^{-i2^k(y_{2n}\theta_{2n} + \bar y\bar \theta)} dy_{2n} d\bar y \]
to write
\[\cA^{k,\ell} f(x,t)=2^{2k}\int
e^{i2^k(\theta_{2n}S^{2n}(x,t,y') +\bar \theta \bar S (x,t,y'))} a_\ell(x,t,y',\theta) F_k(y',\theta_{2n},\bar\theta) dy'd\theta.
\]
By Plancherel's theorem
\Be\label{Plancherel} \|F_k\|_2=2^{-k}2\pi\|f\|_2.
\Ee
Note that $a_\ell$ is supported on a set where $|y'|$ is small and $|\theta|\approx 1$. We make a finite decomposition of the symbol
$a_{\ell} =\sum_i a_{\ell,i}$ where each $a_{\ell,i} $ is supported on a set of diameter $O(\epsilon)$.
It will be convenient to rename the variables $(y',\theta)=(w',w_{2n} , \bar w)$ and replace $F_k(y',\theta_{2n}, \bar \theta)$ by a general function $w\to f(w)$. We are therefore led to consider
the oscillatory integral operator $T^{k,\ell}$ defined by
\[T^{k,\ell} f(x,t) := \int_{\R^d} e^{i 2^k \Phi(x,t,w)} b_{\ell}(x,t,w) f(w) dw,\]
with the phase function
\Be\label{eq:Phi-phase} \Phi(x,t,w)=
w_{2n}S^{2n}(x,t,w') +\bar w \bar S (x,t,w'),
\end{equation}
and symbol $b_\ell$
which is a placeholder for one of the $a_{\ell,i}$. Thus we have \[b_\ell(x,t,w)=\chi_1(x,t,w')\zeta_1(2\bar w) u_\ell(x,w_{2n}, \bar w) \] with $u_\ell$ as in \eqref{uell}. $b_{\ell}$ is smooth and supported in a set of diameter $O(\epsilon)$ where $|w'|\lc\epsilon$, $|x'|\lc\epsilon$, $|x_{2n}-t|\lc \epsilon$, $|\ox |\lc \epsilon$, $|(w_{2n}, \bar w)|\sim 1$ and where $|t-t^\circ|\lc \ep$ for some fixed $t^\circ$, and finally the size of
\begin{equation}
\label{eqn: OIO sigma def}
\sigma (x,w_{2n}, \bar w)=w_{2n}+\bar w\, x^\intercal Je_{2n}
\end{equation}
(i.e. $\sigma$ as in \eqref{sigmadefinition}) is about $2^{-\ell}$.

In view of \eqref{Plancherel} we see that \eqref{eqn:Q5single} follows from
\Be \label{eqn:Q5OIOsingle1}
\| T^{k,\ell} \|_{L^2(\bbR^{2n+1})\to L^{q_5, \infty}(\R^{2n+1}\times \cZ_{k-\ell})} \lc
2^{-k\frac{2n+1}{q_5} } 2^{\ell b_1}.
\Ee
We remark that for $2^\ell\lc \ep^{-1}$ the estimate follows by the consideration in \cite{RoosSeegerSrivastava}, indeed then we can apply a theorem about oscillatory integrals with Carleson-Sj\"olin conditions
(see \cite{SteinBeijing}, \cite{MSS92}). However in view of the properties of the amplitude function $b_\ell$ for large $\ell$ these theorems are no longer directly applicable. In what follows we shall only treat the case for large $\ell$.

In order to show \eqref{eqn:Q5OIOsingle1} it will be convenient to work with a subset of $\cZ_{k-\ell}$ with some additional separation condition. Given small $\nu$
such that \Be\label{eq:nu} 0< \nu <\tfrac 12 \big (b_1- \tfrac{n(1-\gamma)}{2(n+\gamma)} \big)
\Ee we replace $\cZ_{k-\ell}$ with an arbitrary subset $\Zsep$ satisfying the separation condition
\Be\label{Zsep-cond}
t, \breve t\in \Zsep,\, t\neq \breve t \implies |t-\breve t|>2^{\ell-k} 2^{\ell\nu} .
\Ee
It is clear that $\cZ_{k-\ell}$ can be written as a disjoint family of sets $\cZ_{k-\ell, i}$, for $i=1,\dots, N$ with $N\le 2^{1+\ell \nu }$, where each $\cZ_{k-\ell,i}$ satisfies the condition \eqref{Zsep-cond}. By Minkowski's inequality it is therefore enough to prove
\Be \label{eqn:Q5OIOsingle}
\| T^{k,\ell} \|_{L^2(\bbR^{2n+1})\to L^{q_5, \infty}(\R^{2n+1}\times \Zsep)} \lc
2^{-k\frac{2n+1}{q_5} }2^{\ell \frac{n(1-\ga) +\nu \gamma}{2 (n+\gamma)} }
\Ee
for any subset $\Zsep$ of $\cZ_{k-\ell}$ satisfying \eqref{Zsep-cond}.
In what follows we fix such a subset $\Zsep$. We define the operator $S^{k,\ell}$ acting on functions $g:\R^{2n+1}\times \Zsep\to \C$ by
\[ S^{k,\ell}g(x,t) = \sum_{ t'\in\Zsep} T^{k,\ell}_{t} (T^{k,\ell}_{t'} )^* [g(\cdot, t')](x) ,\]
where $T^{k,\ell}_t f(x) = T^{k,\ell} f(x,t)$.
By a $TT^*$ argument, \eqref{eqn:Q5OIOsingle} is a consequence of the following estimate
\begin{equation}
\label{eqn:skl}
\| S^{k,\ell} g \|_{L^{q_5, \infty}(\bbR^{2n+1}\times \Zsep)} \lc
2^{-2k\frac{2n+1}{q_5} }2^{\ell\frac{n(1-\gamma) +\nu\gamma}{n+\gamma}} \| g \|_{L^{q_5', 1}(\bbR^{2n+1}\times \Zsep)}.
\end{equation}
For $j>0 $ and $t\in\Zsep$, we define
\[ \Zsepj(t) = \{ t'\in \Zsep\,:\,2^{(1+\nu)\ell-k+j}\le |t-{t'}|\le 2^{(1+\nu)\ell-k+j+1} \}, \] and for
$j=0$ we set $\cZ_{k-\ell}^0(t)=\{t\}.$
Note that $\Zsepj(t)$ is empty, if $j>k-\ell+4$. Let
\[ S^{k,\ell}_j g(x,t) = \sum_{t'\in\Zsepj(t)} T^{k,\ell}_t (T^{k,\ell}_{t'})^* [g(\cdot, t')](x) \] and observe that \[S^{k,\ell}=\sum_{j\ge 0} S^{k,\ell}_j.\] We claim that $S^{k,\ell}_j$ satisfies for $2\le q\le \infty $ the estimates
\begin{multline}
\label{eqn: skl Lq'Lq}
\|S^{k,\ell}_jg \|_{L^q(\mathbb{R}^{2n+1}\times\Zsep) }
\\\lesssim 2^{-k\frac{4n+2}{q}}
2^{\ell((n-\nu+1)
\frac 2q-(n-\nu))} 2^{j(\frac{2(n+\gamma)}q -n)} \|g\|_{L^{q'}(\mathbb{R}^{2n+1}\times\Zsep)},
\end{multline}
which follow by interpolation from
\begin{equation}
\label{eqn: sklj L2}
\|S^{k,\ell}_jg\|_{L^2(\mathbb{R}^{2n+1}\times\Zsep) }\lesssim 2^{-k(2n+1)}2^{\ell}2^{j\gamma}\|g\|_{L^2(\mathbb{R}^{2n+1}\times
\Zsep)}
\end{equation}
and
\begin{equation}
\label{eqn: skl L1Linf}
\|S^{k,\ell}_jg \|_{L^\infty(\mathbb{R}^{2n+1}\times\Zsep) }\lesssim 2^{-\ell(n-\nu)} 2^{-jn} \|g\|_{L^1(\mathbb{R}^{2n+1}\times\Zsep)}.
\end{equation}

Clearly, if $q>q_5=\frac{2(n+\gamma)}{n} $ we can sum in $j$ in \eqref{eqn: skl Lq'Lq} to get
\begin{multline}
\label{eqn: skl Lq'Lqsum}
\|S^{k,\ell}g \|_{L^q(\mathbb{R}^{2n+1}\times\Zsep) }
\\\lesssim_q 2^{-k\frac{4n+2}{q}}
2^{\ell((n-\nu+1)
\frac 2q-(n-\nu))} \|g\|_{L^{q'}(\mathbb{R}^{2n+1}\times\Zsep)}, \quad q>q_5.
\end{multline}
Moreover for $q=q_5$ we can apply Bourgain's interpolation trick to obtain the restricted weak type inequality
\eqref{eqn:skl}.

To prove \eqref{eqn: sklj L2} we estimate
\begin{align*}
&\|S^{k,\ell}_jg\|_{L^2(\mathbb{R}^{2n+1}\times\Zsep)}
\\&\quad=\Big(\sum_{t\in \Zsep}\int\Big|\sum_{t'\in\Zsepj(t)} T^{k,\ell}_{t} (T^{k,\ell}_{t'})^* [g(\cdot, t')](x) \Big|^2\,dx\Big)^{1/2}\\
&\quad\leq \Big(\sum_{t\in \Zsep}\# \Zsepj(t) \int\sum_{t'\in\Zsepj(t)} |T^{k,\ell}_{t} (T^{k,\ell}_{t'})^* [g(\cdot, t')](x)|^2\,dx\Big)^{1/2}\\
&\quad\lc \Big(\sum_{t\in \Zsep }
\#\Zsepj(t) \sum_{t'\in \cZ^j_{k-\ell} (t)}
\|T^{k,\ell}_t\|^2_{L^2\to L^2}
\|(T^{k,\ell}_{t'})^*\|^2_{L^2\to L^2} \|g(\cdot, t') \|_2^2\Big)^{1/2}
\\
&\quad \lesssim A_2 2^{-k(2n+1)}2^{\ell}2^{j\gamma}\|g\|_{L^2(\mathbb{R}^{2n+1}\times\Zsep)}.
\end{align*}
Here we have used the fact that \[\|T^{k,\ell}_{t}\|_{L^2\to L^2}\lesssim 2^{-k}\|\mathcal{A}^{k,\ell}_{t}\|_{_{L^2\to L^2}}\lc 2^{ \frac \ell 2} 2^{-k (n+\frac 12) }\] and that $\#\Zsepj(t')\lesssim A_2 2^{j\gamma}$ for all $t'\in\Zsepj$. This takes care of \eqref{eqn: sklj L2}.

Inequality \eqref{eqn: skl L1Linf} is a direct consequence of the following kernel estimate, which shall be proved in \S\ref{sec: kernelest}.
\begin{prop}\label{prop:kernelest}
Let $k>0$ and $0\le \ell\le [k/3]$.
Let $\mathcal{K}^{k,\ell}_{t,\breve{t}}$ denote the kernel of $T^{k,\ell}_t (T^{k,\ell}_{\breve{t}})^*$, which is given by
\Be\label{Kkell} \mathcal{K}^{k,\ell}_{t,\breve{t}}(x,\breve{x}) = \int_{\R^{2n+1}} e^{i 2^k (\Phi(x,t,w)-\Phi(\breve{x},\breve{t},w) )} b_\ell(x,t,w) \overline{b_\ell(\breve{x},\breve{t},w)} \, dw. \Ee
Let $\nu>0$ and \Be\label{eq:ell-restr} 2^{\ell-k}2^{\nu \ell}\le |t-\breve{t}|\le 1.\Ee Then
for $0\leq \ell\le [\tfrac{k}{3}]$, we have
\[ |\mathcal{K}^{k,\ell}_{t,\breve{t}}(x,\breve{x})| \lesssim_\nu 2^{\ell\nu} (1+2^k|t-\breve{t}|)^{-n}. \]

\end{prop}

\begin{rem}
One can run the above arguments also for $n=1$. A favorable $L^2\to L^q$ bound for $\cA_{k,\ell}$ follows if
$q>2(1+\gamma)$ because then the $j$-sum of the terms in
\eqref{eqn: skl Lq'Lq} converges for the case $n=1$ of \eqref{eqn: skl Lq'Lqsum}. The exponent of $2^\ell$ in \eqref{eqn: skl Lq'Lqsum} is now positive for all $\nu>0$ when $q<4$, and we have to allow the range $\ell\le k/3$. Thus we get a positive result when
$-\tfrac 6q + \tfrac 13(\tfrac 4q-1) < -2$ which is the case for $q<14/5$. This restricts the range of allowable $\gamma$ to $2(1+\gamma)<14/5$, i.e. $\gamma<2/5$. As a result one obtains that $M_E$ maps $L^2(\bbH^1)$ to $L^q(\bbH^1)$ if $\dim_{\mathrm {qA}}E<2/5$ and $q<14/5$. We know from considerations in \cite{GreenleafSeeger1994, RoosSeegerSrivastava} that this result is not sharp; this point will be addressed elsewhere.
\end{rem}

\section{Proof of Proposition \ref{prop:kernelest}}
\label{sec: kernelest}
In order to estimate the oscillatory integral \eqref{Kkell} using stationary phase arguments we expand the phase $\Phi(x,t,w)-\Phi(\breve x,\breve t,w) $ as
\[
(x-\breve x)^\intercal \nabla_x\Phi(\breve x,\breve t,w) +(t-\breve t)\partial_t\Phi(\breve x,\breve t,w)+ O(|(x-\breve x, t-\breve t)|^2)\] and thus, for stationary phase calculations it is natural to consider the curvature property of the surface
\[\Sigma_{x,t}=\{ \nabla_{x,t} \Phi (x,t,w) \} \]
where $w$ is close to a reference point $w^\circ$ with $(w')^\circ=0$.
These considerations are similar to those in the proof of Stein's result on Carleson-Sj\"olin type oscillatory integral operators (see \cite{SteinBeijing, Stein-harmonic} and also \cite{MSS93}). A potential difficulty here is that for large
$\ell$ and small $|x-\breve x|+|t-\breve t|$ the amplitudes do not a priori seem to satisfy the appropriate derivative bounds for an application of the stationary phase method.
However, a closer examination of the curvature properties of $\Sigma_{x,t}$ and their interplay with the geometry of the fold surface $\{\sigma=0\}$ will reveal that this is not a significant obstacle in our specific situation.

\subsection{\texorpdfstring{Curvature of $\Sigma_{x,t} $}{Curvature}}\label{sec:curvature}
We analyze the $w$-derivatives of
\Be \label{eq:Xidef}
\Xi(x,t,w):=\nabla_{x,t} \Phi (x,t,w)=
w_{2n} \nabla_{x,t}S^{2n} (x,t,w') + \barw \nabla_{x,t}\bar{S} (x,t,w'),
\Ee
for a fixed $(x,t)$. These calculation will be the basis for a stationary phase estimate in \S \ref{sec:contproof}. We will only consider the case of large $\ell$ i.e. when \Be\label{eq:sigma}\sigma\equiv\sigma(x,w_{2n},\ow)= w_{2n} +\ow\ux^\intercal Je_{2n}
\Ee is small ($|\sigma| \lc 2^{-\ell} $) since the other cases have already been discussed in \cite{RoosSeegerSrivastava}. We need some modifications because of the lack of good differentiability properties of the amplitudes for large $\ell$.

For the sake of completeness, we include the calculation of the curvature matrix below, and then establish the invertibility of this minor. Using \eqref{eq:Xidef}, the expressions for $S^{2n}, \bar{S}$, and the skew-symmetry of $J$ we calculate that $\Xi(x,t,w)$ is equal to
\[
w_{2n} \begin{pmatrix}
-\nabla g(\ttfw) \\1\\ 0\\ g_*(\ttfw)
\end{pmatrix}
\,+\,\bar{w}\begin{pmatrix}
PJP^\intercal w'
-tg(\ttfw) PJe_{2n}
-(\ux^{\intercal} J e_{2n}) \nabla g(\ttfw)
\\
e_{2n}^\intercal J P^\intercal w'
\\1
\\
g_*(\ttfw)\ux^{\intercal}Je_{2n}
\end{pmatrix}
\]
where
\begin{subequations}\label{hdef}
\Be \label{eq:defofh} g_*(x')=\inn{x'}{\nabla g(x')}-g(x'),\Ee
with
\Be\label{hderiv} g_*(0)=-1, \quad \nabla g_*(0)=0, \quad g_*''(0)= -I_{2n-1}.\Ee
\end{subequations}

The oscillatory integral operator $f\mapsto T^k f(\cdot,t):=\sum_{\ell}T^{k,\ell} f(\cdot,t)$ is an operator with a folding canonical relation (i.e. two-sided fold singularities), and the fold surface is parametrized by $\sigma=0$ (see \cite[Remark 3.2]{RoosSeegerSrivastava}, \cite{MuellerSeeger2004} and the discussion after \eqref{Littlewood-Paley} in the analogous setting of Fourier integral operators, for more details).

We compute, for $j=1,\dots, 2n-1$, the partial derivatives
(recalling the expression for $\sigma$ from \eqref{eq:sigma}),
\[
\Xi_{w_j}=\begin{pmatrix}
t^{-1}\sigma\partial_j\nabla g(\ttfw)+\ow PJ( e_j+\partial_jg(\ttfw) e_{2n}) \\ \ow e_{2n}^\intercal J e_j
\\
0
\\ -t^{-1}\sigma\partial_jg_*(\ttfw)
\end{pmatrix},\]
\[
\Xi_{w_{2n}}=\begin{pmatrix}
-\nabla g(\ttfw)\\1\\0\\g_*(\ttfw)
\end{pmatrix},\] and, with $\barw\equiv w_{2n+1}$,
\[\Xi_{{w}_{2n+1}}=
\begin{pmatrix}
PJP^\intercal w'
-tg(\ttfw) PJe_{2n}
-\ux^{\intercal} J e_{2n} \nabla g(\ttfw)
\\ e_{2n}^\intercal J P^\intercal w'
\\1
\\
g_*(\ttfw)\ux^{\intercal}Je_{2n}
\end{pmatrix}.
\]

For $x'=w'$,
using the properties of $g,h$ in
\eqref{eq:Psidefc}, \eqref{hderiv} we get
\begin{align*}
\Xi_{w_j}
\Big|_{x'=w'}
&=(-t^{-1}\sigma +\ow J)e_j,\\
\Xi_{w_{2n}} \Big|_{x'=w'}&=\begin{pmatrix}
\vec{0}_{2n-1}\\1\\0\\-1
\end{pmatrix},\,\, \Xi_{{w}_{2n+1}} \Big|_{x'=w'}=\begin{pmatrix}
PJP^\intercal w' -t PJe_{2n}\\
e_{2n}^\intercal J P^\intercal w'\\1\\ -\ux^\intercal Je_{2n}
\end{pmatrix}.
\end{align*}

Using the defining equations of a unit normal vector $N$, \[\inn {N}{ \Xi_{w_i} }=0 ,\,\, i=1,\dots, 2n+1 \]
at the north pole ($x'=w'$), we get

\begin{subequations}
\begin{align} \label{normal1}
0&=\inn {N}{\Xi_{w_j} }\Big|_{x'=w'}=-t^{-1}\sigma
\alpha_j+\ow\underline{\alpha}^{\intercal}Je_j, \quad j\le 2n-1.
\\
\label{normal2} 0&=\inn{N} {\Xi_{w_{2n} }}\Big|_{x'=w'}=\alpha_{2n}-\alpha_{2n+2},
\end{align} and
\begin{multline}
\label{normal3}
0=\inn {N} {\Xi_{w_{2n+1}} }\Big|_{x'=w'}=\\ \,
\alpha'^{\intercal}(PJP^\intercal w' -t PJ e_{2n})
+\alpha_{2n} e_{2n}^\intercal J P^\intercal w' +\alpha_{2n+1}
-\alpha_{2n+2}x^\intercal Je_{2n},
\end{multline}
\end{subequations}
where $N^\intercal=(\alpha'^\intercal,\alpha_{2n},\bar \alpha)$.
Equation \eqref{normal3} above expresses ${\alpha}_{2n+1}$ in terms of $\underline{\alpha}$ and $\alpha_{2n+2}$ and turns out to be not really relevant to our calculations. Since $|N|=1$ we have $|\underline\alpha|\approx 1$.

The second derivative vectors are given by
\[
\Xi_{w_jw_k}=\begin{pmatrix}
-t^{-2}\sigma\partial_{jk}\nabla g(\ttfw)-\ow t^{-1}PJe_{2n}\partial_{jk}^2g(\ttfw) \\0\\0\\ t^{-2}\sigma\partial_{jk}^2g_*(\ttfw)
\end{pmatrix},\] for $1\le j,k\le 2n-1$,
and
\[ \Xi_{w_{j}w_{k}} = 0, \text{ if } 2n\le j, k\le 2n+1.
\]
Moreover, for $j=1,\dots, 2n-1$,
\[
\Xi_{w_jw_{2n}}=\begin{pmatrix}
t^{-1}\partial_j\nabla g(\ttfw)\\0\\0\\-t^{-1}\partial_jg_*(\ttfw)
\end{pmatrix},\] and,
\[\Xi_{w_j{w}_{2n+1}}=\begin{pmatrix}
PJe_j+ PJe_{2n}\partial_jg(\ttfw)+t^{-1}\ux^{\intercal} Je_{2n}\partial_j\nabla g(\ttfw) \\ e_{2n}^{\intercal}J e_j\\0\\
- t^{-1}\ux^{\intercal}Je_{2n}\partial_j g_*(\ttfw)
\end{pmatrix}.
\]
\detail{
For $x'=w'$ we get (using
$g''(0)=g_*''(0)=-I_{2n-1}$, $g'''(0)=0$)
for $1\le j,k\le 2n-1$,
\begin{align*}
\Xi_{y_jy_j} \Big|_{x'=y'}&=\begin{pmatrix}
t^{-1}PJ_{\overline{y}}e_{2n}\ \\0\\\vec{0}_m\\ -t^{-2}
\sigma- t^{-1}\La_{\oy} e_{2n}
\end{pmatrix},\\ \Xi_{y_j y_k}\Big|_{x'=y'} &=0 \text{ if } j\neq k, \qquad
\Xi_{y_jy_{2n}} \Big|_{x'=y'}=
-t^{-1}e_j,\\
\Xi_{y_j{y}_{2n+i}} \Big|_{x'=y'}&=\begin{pmatrix}
J_ie_j \\\vec{0}_{m}\\-\La_ie_j
\end{pmatrix}-t^{-1}((\ux^{\intercal} J_i -t\La_i)e_{2n})e_j
\end{align*}
and
\[ \Xi_{y_{2n}y_{2n}} = \Xi_{\overline{y}_i\overline{y}_i'}=\Xi_{y_{2n}\overline{y}_i} =\vec 0 \text{ when $x'=y'$.}\]
}
We evaluate at $x'=w'$, using $g''(0)=g_*''(0)=-I_{2n-1}$, $g'''(0)=0$, and see that the components
of the curvature matrix $\mathscr C^N$ at $x'=w'$
are given by
\begin{align*}
\inn{N}{\Xi_{w_jw_j}}\Big|_{x'=w'}&=t^{-1}(\alpha')^\intercal PJe_{2n}\ow -t^{-2}\alpha_{2n+2}\sigma,\\
\inn{N}{\Xi_{w_jw_k}}\Big|_{x'=w'}&=0,\quad \text {if } j\neq k,
\end{align*}
for $1\le j,k\le 2n-1$. Moreover for $1\le j\le 2n-1$,
\begin{align*}
\inn{N}{\Xi_{w_jw_{2n}}}\Big|_{x'=w'}&=-t^{-1}\alpha_j,\\
\inn{N}{\Xi_{w_j{w}_{2n+1}}}\Big|_{x'=w'}&=\ubar{\alpha}^{\intercal} J e_j- t^{-1}\alpha_j\ux^{\intercal}Je_{2n},
\end{align*}
and
\[ \inn{N}{\Xi_{w_{j}w_{k}}}\Big|_{x'=w'}= 0, \quad j,k\in \{2n, 2n+1\}.\]
Thus, the curvature matrix $\mathscr C^N$ at $x'=w'$ with entries
$\inn {N}{\Xi_{w_iw_j} }$, $1\le i,j\le 2n+1$ (with $w_{2n+1}\equiv \ow$) is
\[ \mathscr C^N =\begin{pmatrix}
c\mathrm{I}_{2n-1} &PA\\A^\intercal P^\intercal &0
\end{pmatrix} \Big|_{x'=w'},
\]
where the scalar $c$ and the $2n\times 2$ matrix $A$ are given by
\begin{align}
\label{eqn: cdef}
c&=\tfrac{\underline\alpha^\intercal Je_{2n}}{t} \ow - \tfrac{ \alpha_{2n+2}\sigma}{t^2}
\\
A&=
\begin{pmatrix} -\frac 1t \underline\alpha & J\underline\alpha-\frac{ \underline x^\intercal Je_{2n} }{t} \underline\alpha
\end{pmatrix}\nonumber
\end{align}
(and $PA$ is the $(2n-1)\times 2$ matrix obtained by deleting the last row of $A$). Using \cite[Lemma 3.1]{RoosSeegerSrivastava} and the fact that $|\underline \alpha|\approx 1$, it can be checked that $|c|$ is uniformly bounded away from zero, which implies that the rank of the curvature matrix is $2n$ (indeed by \eqref{normal1} we have $PJ\underline{\alpha}=0$ when $x'=w'$ and $\sigma=0$, hence $\rank(PA)=1$).

As a consequence of the above we obtain for the restricted matrices
\begin{align} \label{partialcurvature} \det \begin{pmatrix} D^2_{ \uw\,\uw} \inn{\Xi}{N}\end{pmatrix}\Big|_{x'=w'} &=- c^{2n-2}\frac{|\alpha'|^2}{t^2} \neq 0
\\
\label{restrpartialcurvature} \det \begin{pmatrix} D^2_{w'w'}\inn {\Xi}{N}\end{pmatrix}\Big|_{x'=w'} &=c^{2n-1} \neq 0
\end{align}
with $c$ as in \eqref{eqn: cdef}.

\subsection{Proof of Proposition \ref{prop:kernelest}, continued}\label{sec:contproof}
Recall that
\[ \mathcal{K}^{k,\ell}_{t,\breve{t}}(x,\breve{x}) = \int e^{i 2^k (\Phi(x,t,w)-\Phi(\breve{x},\breve{t},w))} b_\ell(x,t,w) \overline{b_\ell(\breve{x},\breve{t},w)} \, dw. \]
For ease of notation, we set \[X=(x,t), \,\,\Breve{X}=(\Breve{x}, \Breve{t}),\]
and
\[B_\ell(X,\Breve{X}, w)=b_\ell(x,t,w) \overline{b_\ell(\breve{x},\breve{t},w)}.\]
Recall that the amplitude $B_\ell$ is supported in the set where \begin{equation}\label{eqn:amp supp}|(w_{2n}, \ow)|\sim 1,\,|w'|\leq \epsilon,\, |x'|\le\epsilon,\,|\Breve{x}'|\le\epsilon,\, |\ox |\le \epsilon, \, |\Bar{\Breve{x}} |\le \epsilon, \end{equation}
\[|x_{2n}-t|\le \epsilon,\,|\Breve{x}_{2n}-\Breve{t}|\le \epsilon, \, |t-\breve t| \le \epsilon\] and
\[
|w_{2n}+\ow\, \ux^\intercal Je_{2n}|\approx 2^{-\ell}\approx |w_{2n}+\ow\, \underline{\Breve{x}}^\intercal Je_{2n}|.
\]
We fix a reference point $(X^\circ, w^\circ)$ where
\[X^\circ=(0',x_{2n}^\circ, \bar x^\circ, t^\circ), \quad
w^\circ= (0', 0,\barw^\circ) \] (so that $\sigma$ becomes $0$ at $(X^\circ, w^\circ)$, and let $N^\circ $ be one of the unit normals to $\Sigma_{X^\circ} $ at $w=w^\circ$, i.e. we have
\begin{equation}
\label{eqn:N}
\langle N^\circ, \partial_{w_j}\nabla_X \Phi (X^\circ,w^\circ)\rangle=0,\,\, \text{ for } 1\leq j\leq 2n+1.
\end{equation}
Then $B_\ell$ is supported in a ball of radius $O(\ep)$ centered at $(X^\circ, \breve X^\circ, w^\circ)$.

For a unit vector $\vec u$ define
\[ \Psi(X,\breve X, \vec u, w) =
\int_0^1\vec u\cdot \nabla_X \Phi ( \breve X+s(X-\Breve{X}),w) \,ds.
\]
Then
we can express the phase function corresponding to the kernel $\cK_{t,\Breve{t}}^{k,\ell}$ as
\Be\label{eq:ftc}2^k(\Phi(X,y)-\Phi(\breve{X},y))=\lambda\,
\Psi(X,\breve X, \tfrac{X-\breve X}{|X-\breve X|}, w), \text{ with $\lambda= 2^k|X-\Breve{X}|.$}
\Ee

Define for all $\vec u\in S^{2n+1}$
\[ \cI_{\la,\ell} (X,\breve X, \vec u)= \int e^{i\la \Psi(X,\breve X, \vec u, w)} B_\ell(X,\breve X, w) dw
\]
and note that $\Psi$ is a smooth phase, in all arguments.
\begin{lem}\label{lem:oio} Let $\nu>0$. For $\ep$ in \eqref{eqn:amp supp} sufficiently small the following holds, for $2^\ell>\ep^{-1}$, $\la>\ep^{-1}$.

(i) For $\min \{|\vec u-N^\circ|, |\vec u+N^\circ|\} \ge \ep^{3/4}$ we have
\[|\cI_{\la,\ell} (X,\breve X,\vec u)|\le C_{M,\ep} 2^{-\ell}(\la 2^{-\ell})^{-M}.\]

(ii) For $\min \{|\vec u-N^\circ|, |\vec u+N^\circ|\}\le \ep^{1/2} $ and $2^\ell \le \la^{\frac{1}{2(1+\nu)}} $ we have
\[|\cI_{\la,\ell} (X,\breve X,\vec u)|\lc_\ep \la^{-n}.
\]

(iii)
For $\min \{|\vec u-N^\circ|, |\vec u+N^\circ|\} \le \ep^{1/2} $,
we have
\[|\cI_{\la,\ell} (X,\breve X,\vec u)|\lc_\ep 2^{-\ell} \la^{-\frac {2n-1}{2} }.
\]
If in particular $2^\ell \ge \la^{\frac{1}{2(1+\nu)}} $ then
\[|\cI_{\la,\ell} (X,\breve X,\vec u)|\lc_\ep 2^{\ell \nu} \la^{-n}.
\]
\end{lem}

\begin{remarka} The conclusions in part (ii), (iii) also hold for $\nu=0$ but in (ii) require a stationary phase estimate for amplitudes $\chi_\la$ satisfying endpoint Calder\'on-Vaillancourt bounds, i.e. $\partial_w^\alpha (\chi_\la (w)) =O(\la^{|\alpha|/2}) $. For our application it suffices to take $\nu>0$.
\end{remarka}

We first show that Lemma \ref{lem:oio} implies Proposition \ref{prop:kernelest}. We take $X\neq\breve X$ and $\vec u=\frac{X-\breve X}{|X-\breve X|}$, and $\la=2^k|X-\breve X|$. Assume $\min |\frac{X-\breve X}{|X-\breve X|} \pm N_0|\ge \ep^{3/4}$. We have $|t-\breve t|\ge 2^{\ell-k} 2^{\nu\ell}$ and get from part (i) of Lemma \ref{lem:oio} the estimate, for $N\gg n$,
\begin{align*}|\cI_{\la,\ell}|&\lc 2^{-\ell}(\la 2^{-\ell} )^{-N} \lc_N 2^{\ell(n-1)} (2^k |X-\breve X|)^{-n} (2^{k-\ell}|X-\breve X|)^{n-N})
\\&\lc (2^k|X-\breve X|)^{-n} 2^{\ell (n-1- \nu(N-n))}.
\end{align*} The bound $|\cI_{\la,\ell}|\lc (2^k |X-\breve X|)^{-n} $ follows if we choose $N$ large enough.

If $\min |\frac{X-\breve X}{|X-\breve X|} \pm N_0|\le \ep^{3/4}$ the appropriate bound is in part (ii) of the lemma, and the bound in Proposition \ref{prop:kernelest} is now established
for the range $2^\ell\le (2^k|X-\breve X|)^{\frac{1}{2(1+\nu)} } $, i.e. $|X-\breve X| \ge 2^{2\ell(1+\nu)-k}$.

Next assume $t\neq \breve t$,
$|X-\breve X| \le 2^{2\ell(1+\nu)-k}$, by the assumed $t$-variation we also have the lower bound
and $|X-\breve X|\ge 2^{\ell(1+ \nu)-k}$ which is needed to apply part (i) of Lemma \ref{lem:oio} for $\min |\frac{X-\breve X}{|X-\breve X|} \pm N_0|\ge \ep^{3/4}$. In the opposite range we apply part (iii) of the lemma. Note that the assumption
$2^\ell \ge \la^{\frac{1}{2(1+\nu)}} $ is now equivalent to the required $|X-\breve X|\ge 2^{2\ell(1+\nu)-k} $.
We also note that $\partial_{w'}^\alpha(B_\ell(X,\breve X,w))=O(1) $.

This finishes the proof of Proposition \ref{prop:kernelest} once Lemma \ref{lem:oio} is verified.

\subsection{Proof of Lemma \ref{lem:oio}}
Let $V$ be the linear space perpendicular to $N^\circ$; then $\nabla^2_{(X,w) }\Phi $ is invertible as a map from $\bbR^{2n+1}$ to $V$. Hence
\[ \big| \nabla_w \inn{\vec u}{\nabla_{x,t} \Phi(X^\circ,w)}_{w=w^\circ} \big|\ge c|\vec u-\inn{\vec u}{N^\circ} N^\circ | \gc \ep^{3/4}, \]
and by expanding $\nabla_w\Psi (X,\breve X,\vec u, w)$ about $(X^\circ, X^\circ, \vec u, w^\circ) $ we get
\[
\nabla_w \Psi(X,\breve X,\vec u, w) - \nabla_w \inn{\vec u}{\nabla_{x,t} \Phi(X^\circ,w)}\big|_{w=w^\circ} =O( \ep).
\]
This implies that for $|\vec u-N^\circ |\ge\ep^{3/4}$ and $\ep$ small
\[\big|
\nabla_w \Psi(X,\breve X,\vec u, w)\big| \gc\ep^{3/4}
\]
for $(X,\breve X,w)$ in the support of $B_\ell$. Since the higher $w$-derivatives of $\Psi$ are bounded and since
\begin{equation}
\label{eqn:blowup est-y}
\partial_{(w_{2n},\ow )} ^\alpha \big[ B_\ell(x,t,w)]=O(2^{\ell|\alpha|} )
\end{equation}
an integration by parts yields the bound $\cI_{\la,\ell}= O(2^{-\ell}(\la 2^{-\ell})^N)$ as asserted.

We now turn to (ii) and apply a stationary phase argument with respect to the
$\uw$-variables. By our curvature calculations
the $(2n\times 2n)$ Hessian matrix $D^2_{\uw\,\uw} \big(\inn{N^\circ}{\nabla_{X} \Phi (X^\circ, \uw,\barw)}\big)_{w=w^\circ} $ is invertible, for $|u-N^\circ| \le \ep^{1/2} $ we get a matrix norm estimate
\[\big\|
D^2_{\uw\,\uw} (\inn{N^\circ}{\nabla_{X} \Phi (X^\circ, w)})_{w=w^\circ} - D^2_{\uw\,\uw } \Psi(X,\vec u, w) \big\|\lc \ep^{1/4}
\] and hence (given that $\ep$ is small) we see that
$D^2_{\uw\,\uw } \Psi(X,\vec u, w)$ is invertible, with uniformly bounded inverse. Note that
by our assumption on $\ell$ and $\la$ we have $\partial^\alpha_{\uw} B_\ell(X,\breve X,w) = O( \la^{|\alpha|/(2+2\nu)})$ and so for $\nu>0$ a standard application of the stationary phase method in the $\uw$ variables gives the estimate $|\cI_{\la,\ell}|= O(\la^{-n})$.

For (iii) we argue similarly but in view of the unfavorable differentiability properties of $B_\ell$ with respect to $w_{2n}$ we are freezing both the $w_{2n}$ and $\barw$ variables.
We now have that the
$(2n-1)\times (2n-1)$ Hessian matrix $D^2_{w'w'} \big(\inn{N^\circ}{\nabla_{X} \Phi (X^\circ, w',w_{2n},\barw)}\big)_{w=w^\circ} $ is the identity matrix and by a perturbation argument as above we see that
$D^2_{w'w'} \Psi(X,\vec u, w)$ is invertible.
Since $\sigma$ does not depend on $w'$ we have uniform upper bounds for the $w'$-derivatives of the amplitude. We can therefore apply the method of stationary phase in the $w'$-variables and since the $(w_{2n}, \barw)$-integral is extended over a set of measure $O(2^{-\ell})$ we obtain the asserted estimate $|\cI_{\la,\ell}|=O(2^{-\ell} \la^{-\frac{2n-1}{2}})$.
The second estimate in (iii) is immediate since the inequality
$2^\ell \ge \la^{\frac{1}{2(1+\nu)}} $ is equivalent with
$2^{-\ell} \la^{-\frac{2n-1}{2}} \le 2^{\ell\nu} \la^{-n}$.
\qed

\section{Necessary Conditions}\label{sec:necessary}
In this section we prove the sharpness of Theorem \ref{thm:main} for Assouad regular sets $E$. Regarding the line connecting $Q_1$ and $Q_{2,\beta}$ this is just the necessary condition $p\leq q$ imposed by translation invariance and noncompactness of the group $\bbH^n$.
The necessary conditions for the segments $\overline{Q_{2,\beta}, Q_{3,\beta}} $ and $\overline{Q_1Q_{4,\gamma}}$ are quite similar to the consideration in the Euclidean case. However the example
for the segment $\overline{Q_{3,\beta}Q_{4,\gamma}}$ is substantially different from a Knapp type example for co-dimension two surfaces in the Euclidean case (see also \cite{RoosSeegerSrivastava} for a simplified version for the full maximal operator); this indicates a new phenomenon on the Heisenberg group.

Given $\delta\in(0,1)$, let $\cI_{\delta}(E)$ denote the set of all dyadic intervals of the form $[\nu\delta, (\nu+1)\delta)$ (with $\nu\in \mathbb{Z}$) which intersect $E$, and let $\cZ_\delta(E)$ denote a subset of $E$ which contains exactly one $t\in E\cap I$ for every $I\in \cI_\delta(E)$.
Let $\beta=\dim_{\rmM}E$, and $\gamma=\dim_{\mathrm{qA}}E$, respectively.

\subsection{\texorpdfstring{The line connecting $Q_{2,\beta}$ and $Q_{3,\beta}$}{The line connecting Q2 and Q3}} \label{sec:Q2Q3}
For any $\eps>0$ there exists a set $\Delta_{\eps}=\{\delta_j: j=1,2,\dots\}$ with $\lim_{j\to \infty} \delta_j=0$ such that
$N(E,\delta)\ge \delta^{-\beta+\eps}$ for $\delta\in \Delta_{\eps}$.
For $\delta\in \Delta_\eps$ let $f_{\delta}$ be the characteristic function of $B_{10\delta}$, the ball of radius $10\delta$ centered at the origin. Then \[\|f_{\delta}\|_p\approx \delta^{(2n+1)/p}.\]
For $1\leq t\leq 2$ we consider the sets
\[ R_{\delta,t}:= \{(\ubar{x}, \bar x) : ||\ubar x|-t|\le \delta/20,\,|\bar x|\le \delta/20\}.\]
Then $|R_{\delta,t} |\gc \delta^{2} $.
Let $\Sigma_{x,t}= \{\om \in S^{2n-1}: |\ubar x-t \om|\le \delta/4\} $ which has spherical measure $\approx \delta^{2n-1}$.

If $x\in R_{\delta,t}$ and $\om \in \Sigma_{x,t}$ then $|\ubar x-t\om |\le \delta$ and using the skew symmetry of $J$ we get
\[|\bar x- t \ubar x^\intercal J \om| \le
|\bar x|+|\ubar x^\intercal J(t \om - \ubar x)| \le 3\delta.\]
Thus, for $x\in R_{\delta,t}$, \[f_{\delta}*\mu_t(\ubar{x},\bar{x})=\int_{S^{2n-1}} f_{\delta}(\ubar{x}-t{\om},\bar{x}-t\ubar{x}^\intercal\! J{\om})\,d\mu({\om})\gtrsim \delta^{2n-1}.\]

Passing to the maximal operator, we set
\[R_\delta=\cup_{t\in \cZ_{\delta}(E)} R_{\delta,t}. \]
We have $| R_\delta|\gc\delta^2N(E,\delta)\gtrsim \delta^{2+\eps-\beta}$. Further, for $x\in R_{\delta}$, there exists a unique $t(x)\in \cZ_{\delta}(E)$ such that $|f_\delta*\mu_{t(x)}(x)|\ge \delta^{2n-1}$.

This yields the inequality
\[\delta^{2n-1}\delta^{(2+
\eps-\beta)/q}\lesssim \delta^{(2n+1)/p}.\]
We set $\delta=\delta_j$ and let $j\to \infty$, and since $\eps>0$ was arbitrary we
obtain the necessary condition
\begin{equation}
\label{ball ineq}
\tfrac{2-\beta}{q}+2n-1\geq \tfrac{2n+1}{p},
\end{equation}
that is, $(1/p,1/q)$ lies on or above the line connecting $Q_{2,\beta}$ and $Q_{3,\beta}$.

\subsection{\texorpdfstring{The line connecting $Q_1$ and $Q_{4,\gamma}$}{The line connecting Q1 and Q4}} \label{sec:Q1Q4} For this line we just use the counterexample for the individual averaging operators, bounding the maximal function from below by an averaging operator.
Given $t\in [1,2]$, let $g_{\delta,t}$ be the characteristic function of the set
$\{(\ubar{y},\bar{y}): | |\ubar{y}|-t|\le 10\delta, |\bar{y}|\le 10\delta\}$. Thus $\|g_{\delta,t}\|_p\lesssim \delta^{2/p}. $

Let $x=(\ubar{x},\bar{x})$ be such that $|\ubar x|\le \delta$ and $|\bar x| \le \delta$. For any $\om\in S^{2n-1}$, we have that
$t|\ubar{x}^\intercal J{\om}|\lesssim 2\delta$. Thus
\begin{gather*} \big||\ubar{x}-t\om|-t\big|\le 2 \delta,\\
\big|\bar x-t \ubar x^\intercal J\om\big| \le |\bar x|+t|\ux^\intercal J\om| \le 10 \delta \end{gather*}
implying that $|g_{\delta,t}*\sigma_t(x)|\gtrsim 1$.
This yields the inequality $\delta^{(2n+1)/q}\le \delta^{2/p} $ which leads to the necessary condition
\begin{equation}
\label{scaling ineq}
\tfrac {1}{q} \ge \tfrac{2}{2n+1}\cdot \tfrac 1p,
\end{equation} that is, $(1/p,1/q)$ lies on or above the line connecting $Q_1$ and $Q_{4,\gamma}$.

\subsection{\texorpdfstring{The line connecting $Q_{3,\beta}$ and $Q_{4,\gamma}$}{The line connecting Q3 and Q4}}
Here we assume $\beta>0$ (and therefore $\gamma>0$) since $Q_{3,0}=Q_{4,0}$.
By a change of variables, we can assume that \[J=\tfrac{1}{2}\begin{pmatrix}
0 & I_{n} \\
-I_{n} & 0
\end{pmatrix},\]
with $I_n$ being the $n \times n$ identity matrix.

Let $\eps>0$. By the definition of quasi-Assouad regularity there exists a sequence $\{\delta_j\}_{j=1}^\infty$ of positive numbers with $\lim_{j\to \infty}\delta_j=0$ and intervals $I_j\subset [1,2]$ of length $\delta_j^{\theta}$
with $\theta=1-\beta/\gamma$ such that
\Be\label{lower-Assouad}N(E\cap I_j, \delta_j)\ge (\delta_j/|I_j|)^{\eps-\gamma} = \delta_j^{(1-\theta)(\eps-\gamma)}.\Ee

We let $\cP_\eps$ denote the set of pairs $(\delta_j, I_j)$ and fix $(\delta, I)\in \cP_\eps$. Set
\Be\label{sigmajdef}\varsigma=\delta^{(1-\theta)/2}.\Ee
Let $a$ be the right end point of the interval $I$ and let $f$ be the characteristic function of the set
\[\{(\uz,\oz):|z_l'|\lesssim \varsigma, |z_r'|\lesssim \varsigma,||z_n|-a|\lesssim \delta,||z_{2n}|-a|\lesssim \delta,|\oz|\lesssim \delta^{1-\theta}\},\]
where $\uz=(z_l, z_r)\in \bbR^n\times \bbR^n$ and $z_l=(z_l',z_{n})\in \bbR^{n-1}\times \bbR$, $z_r=(z_r',z_{2n})\in \bbR^{n-1}\times \bbR$.
Then \Be\label{fjLp}\|f\|_{p}\lesssim (\varsigma^{2n-2}\delta^{3-\theta})^{1/p}\approx (\delta^{n(1-\theta)+2})^{1/p}.\Ee

For each $t\in [1,2]$, $t<a$ we define the set \[R^{t} _{\delta}:=\{(\ux,\ox):|x_l'|\lesssim \delta \varsigma^{-1}, |x_r'|\lesssim \delta \varsigma^{-1},|\ox|\lesssim \delta^{1+\theta},||(x_n,x_{2n})|+t-a|\lesssim \delta \}.\]
Clearly $\meas(R^t_\delta)\approx (\delta\varsigma^{-1})^{
2n-2}\delta^{2+\theta}$.
Note that there is a constant $C\ge 1$ such that $R^t_\delta$ and $R^{t'}_{\delta}$ are disjoint if $|t-t'|\ge C\delta$.
We choose a covering of $E\cap I$ by a collection $\mathcal{J}$ of pairwise disjoint intervals, each of length $\delta$ and intersecting $E\cap I$.
Let $\tilde \cJ=\{I_\nu\}_{\nu=1}^N$ be a maximal $2C\delta$-separated subset of intervals in $\cJ$. For each $I_\nu$ pick $t_\nu\in I_\nu\cap E$. Then
$R^{t_\nu}_\delta$ and $R^{t_{\nu'}}_{\delta}$ are disjoint if $\nu\neq \nu'$.
Also
\Be \label{cardcJ} N=\# \widetilde \cJ \gc N(E\cap I,\delta).\Ee

We now prove the lower bound
\Be\label{lowermaxbound}
M_E f(\ubar x,\bar x) \gc \delta^{n(1-\theta)}, \text{ for }(\ubar x, \bar x)\in R^{t_\nu}_\delta.
\Ee

To see \eqref{lowermaxbound}, we need the lower bound
\Be\label{convolutionlowerbound}
|f*\mu_t(\ubar x, \bar x)| \gc \delta^{n(1-\theta)} \text{ for $(\ubar x, \bar x)\in R^t_\delta$.}
\Ee To this end
observe that, given $(\ux,\ox)\in R^t_{\delta}$ and for $\om\in S^{2n-1}$ such that
\[|\om_l'|\lesssim \varsigma,|\om_r'|\lesssim \varsigma,|(\om_n,\om_{2n})-\tfrac{(x_n,x_{2n})}{|(x_n,x_{2n})|}|\lesssim \delta^{1-\theta},\] we have
\[|x_l'+t\om_l'| \lesssim \varsigma, \,\,\,\, |x_r'+t\om_r'| \lesssim \varsigma \]
and
\begin{align*}
|\ox+t\ux^\intercal J\om|
&\lesssim |\ox|+\tfrac{t}{2}|x_l'\om_r'-x_r'\om_l'|+\tfrac{t}{2}|x_n\om_{2n}-x_{2n}\om_n|\\
&\lesssim \delta^{1+\theta}+\delta+\Big|x_n\Big(\om_{2n}-\tfrac{x_{2n}}{|(x_n,x_{2n})|}\Big)-x_{2n}\Big(\om_n-\tfrac{x_{n}}{|(x_n,x_{2n})|}\Big)\Big|\\
&\lesssim \delta^{1-\theta} +|x_n|\delta^{1-\theta}+ |x_{2n}|\delta^{1-\theta}\lesssim \delta^{1-\theta}.
\end{align*}
Also for $i=n,2n$, we compute
\begin{align*}
|x_i+t\om _i|^2&=|x_i|^2+t^2|\om_i|^2+2tx_i\om_i\\
&\leq |x_n|^2+|x_{2n}|^2+2t|(x_n,x_{2n})|+2t(x_i\om_i-|(x_n,x_{2n})|)\\
&\leq (|(x_n,x_{2n})|+t)^2+2t|(x_n,x_{2n})|\big(\tfrac{x_i}{|(x_n,x_{2n})|}\om_i-1\big)\\
&\leq (|(x_n,x_{2n})|+t)^2+2t|(x_n,x_{2n})|(|\om_i|-1)\\
&\leq (|(x_n,x_{2n})|+t)^2+2t|(x_n,x_{2n})|(|(\om_n,\om_{2n})|-1)\\
&=(|(x_n,x_{2n})|+t)^2+2t|(x_n,x_{2n})|(\sqrt{1-|\om_l'|^2-|\om_r'|^2}-1).
\end{align*}
As $||(x_n,x_{2n})|+t-a|\lesssim \delta$, we obtain
\begin{align*}
|(|(x_n,x_{2n})|+t)^2-a^2|&\lesssim \delta,\\
|2t(x_n,x_{2n})|\sqrt{1-|\om_l'|^2-|\om_r'|^2}-1|&\lesssim (|t-a|+\delta)(|\om_l'|^2+|\om_r'|^2)\\
&\lesssim (|I|+\delta)\varsigma^2\lesssim \delta,
\end{align*}
where we use $|I|=\delta^{\theta}=\delta\varsigma^{-2}$. This implies
\[||x_i+t\om _i|^2-a^2|\lesssim \delta\]
and hence $||x_i-t\om_i|-a|\lesssim \delta$.
Thus, for $(\ux,\ox)\in R^t_{\delta}$, we have
\[f*\mu_t(\ux,\ox)=\int_{S^{2n-1}} f(\ux+t\om,\ox+t\ux^\intercal J\om)\,d\mu(\om)\gtrsim \varsigma^{(2n-2)}\delta^{(1-\theta)}=\delta^{n(1-\theta)}\] and \eqref{convolutionlowerbound} is proved. Hence \eqref{lowermaxbound} follows.

The lower bound \eqref{lowermaxbound} implies
\begin{align*} \|M_E f\|_q &\ge \Big(\sum_{\nu=1}^N
\delta^{n(1-\theta) q}
\meas (R^{t_\nu} _\delta) \Big)^{1/q} \\ &\gc \delta^{n(1-\theta)} N(E\cap I,\delta)^{1/q} ((\delta\varsigma^{-1})^{
2n-2}\delta^{2+\theta}) ^{1/q}\\
&\gc \delta^{n(1-\theta)} \delta^{\frac 1q((1+\theta)n+(1-\theta)(\eps-\gamma)+1)}.
\end{align*}
Thus we obtain the necessary condition for $L^p\to L^q$ boundedness
\[\delta^{n(1-\theta)} \delta^{\frac 1q((1+\theta)n+(1-\theta)(\eps-\gamma)+1)} \lc \delta^{\frac 1p (n(1-\theta)+2))}.
\]
for all $(\delta,I)\equiv (\delta_j,I_j)\in \cP_\eps$.
Taking the limit as $j\to \infty$ and using that $\eps>0$ can be chosen arbitrarily small we obtain the necessary condition
\[n(1-\theta)+\tfrac{(1+\theta)n-\gamma (1-\theta)+1} {q} \ge \tfrac{n(1-\theta)+2}{p} \] which using $\theta=1-\beta/\gamma$ is rewritten as
\Be\label{necQ3Q4}\tfrac{n\beta}{\gamma}+ \tfrac 1q\big((2-\tfrac {\beta}{ \gamma})n+1-\beta\big)\le \tfrac{1}{p} \big(\tfrac{n\beta}{\gamma}+2\big).\Ee
In the preceding inequality we get equality for the points
$Q_{3,\beta}$, $Q_{4,\gamma}$ in \eqref{quadrilateral}
and thus \eqref{necQ3Q4} expresses that $(1/p,1/q)$ has to lie on or above the line passing to $Q_{3,\beta} $ and $Q_{4,\gamma}$.

\def\MR#1{}
\providecommand{\bysame}{\leavevmode\hbox to3em{\hrulefill}\thinspace}
\providecommand{\MR}{\relax\ifhmode\unskip\space\fi MR }
\providecommand{\MRhref}[2]{
\href{http://www.ams.org/mathscinet-getitem?mr=#1}{#2}
}
\providecommand{\href}[2]{#2}


\begin{thebibliography}{10}

\bibitem{AndersonHughesRoosSeeger}
Theresa~C. Anderson, Kevin Hughes, Joris Roos, and Andreas Seeger,
\emph{${L^p}\to {L^q}$ bounds for spherical maximal operators}, Math. Z.
\textbf{297} (2021), no.~3-4, 1057--1074. \MR{4361901}

\bibitem{BagchiHaitRoncalThangavelu}
Sayan Bagchi, Sourav Hait, Luz Roncal, and Sundaram Thangavelu, \emph{On the
maximal function associated to the spherical means on the {H}eisenberg
group}, New York J. Math. \textbf{27} (2021), 631--675. \MR{4250270}

\bibitem{BeltranGuoHickmanSeeger}
David Beltran, Shaoming Guo, Jonathan Hickman, and Andreas Seeger, \emph{The
circular maximal operator on {H}eisenberg radial functions}, Ann. Sc. Norm.
Super. Cl. Sci. \textbf{23} (2022), no.~2, 501--568. \MR{4453958}

\bibitem{Bourgain85}
Jean Bourgain, \emph{Estimations de certaines fonctions maximales.}, C. R.
Acad. Sci. Paris S\'er. I \textbf{301} (1985), 499--502. \MR{812567}

\bibitem{CarberySeegerWaingerWright1999}
Anthony Carbery, Andreas Seeger, Stephen Wainger, and James Wright,
\emph{Classes of singular integral operators along variable lines}, J. Geom.
Anal. \textbf{9} (1999), no.~4, 583--605. \MR{1757580}

\bibitem{Cuccagna1997}
Scipio Cuccagna, \emph{{$L^2$} estimates for averaging operators along curves
with two-sided {$k$}-fold singularities}, Duke Math. J. \textbf{89} (1997),
no.~2, 203--216. \MR{1460620}

\bibitem{fraser-book}
Jonathan~M. Fraser, \emph{Assouad dimension and fractal geometry}, Cambridge
Tracts in Mathematics, vol. 222, Cambridge University Press, Cambridge, 2021.
\MR{4411274}

\bibitem{fraser-hare-hare-troscheit-yu}
Jonathan~M. Fraser, Kathryn~E. Hare, Kevin~G. Hare, Sascha Troscheit, and Han
Yu, \emph{The {A}ssouad spectrum and the quasi-{A}ssouad dimension: a tale of
two spectra}, Ann. Acad. Sci. Fenn. Math. \textbf{44} (2019), 379--387.
\MR{3919144}

\bibitem{fraser-yu2}
Jonathan~M. Fraser and Han Yu, \emph{{A}ssouad-type spectra for some fractal
families}, Indiana Univ. Math. J. \textbf{67} (2018), no.~5, 2005--2043.
\MR{3875249}

\bibitem{fraser-yu1}
\bysame, \emph{{N}ew dimension spectra: finer information on scaling and
homogeneity}, Adv. Math. \textbf{329} (2018), 273--328. \MR{3783415}

\bibitem{GangulyThangavelu}
Pritam Ganguly and Sundaram Thangavelu, \emph{On the lacunary spherical maximal
function on the {H}eisenberg group}, J. Funct. Anal. \textbf{280} (2021),
no.~3, 108832, 32pp. \MR{4170795}

\bibitem{GreenleafSeeger1994}
Allan Greenleaf and Andreas Seeger, \emph{Fourier integral operators with fold
singularities}, J. Reine Angew. Math. \textbf{455} (1994), 35--56.
\MR{1293873}

\bibitem{hormanderFIO}
Lars H\"{o}rmander, \emph{Fourier integral operators. {I}}, Acta Math.
\textbf{127} (1971), no.~1-2, 79--183. \MR{388463}

\bibitem{LeeLee}
Juyoung Lee and Sanghyuk Lee, \emph{{$L^p-L^q$} estimates for the circular
maximal operator on {H}eisenberg radial functions}, Math. Ann. \textbf{385}
(2023), no.~3-4, 1521--1544. \MR{4566682}

\bibitem{Lee2003}
Sanghyuk Lee, \emph{Endpoint estimates for the circular maximal function},
Proc. Amer. Math. Soc. \textbf{131} (2003), no.~5, 1433--1442. \MR{1949873}

\bibitem{lu-xi}
Fan L\"u and Li-Feng Xi, \emph{Quasi-{A}ssouad dimension of fractals}, J.
Fractal Geom. \textbf{3} (2016), 187--215. \MR{3501346}

\bibitem{MSS92}
Gerd Mockenhaupt, Andreas Seeger, and Christopher~D. Sogge, \emph{Wave front
sets, local smoothing and {B}ourgain's circular maximal theorem}, Ann. of
Math. (2) \textbf{136} (1992), no.~1, 207--218. \MR{1173929}

\bibitem{MSS93}
\bysame, \emph{Local smoothing of {F}ourier integral operators and
{C}arleson-{S}j\"{o}lin estimates}, J. Amer. Math. Soc. \textbf{6} (1993),
no.~1, 65--130. \MR{1168960}

\bibitem{MuellerSeeger2004}
Detlef M\"{u}ller and Andreas Seeger, \emph{Singular spherical maximal
operators on a class of two step nilpotent {L}ie groups}, Israel J. Math.
\textbf{141} (2004), 315--340. \MR{2063040}

\bibitem{NarayananThangavelu2004}
E.~K. Narayanan and S.~Thangavelu, \emph{An optimal theorem for the spherical
maximal operator on the {H}eisenberg group}, Israel J. Math. \textbf{144}
(2004), 211--219. \MR{1448717}

\bibitem{NevoThangavelu1997}
Amos Nevo and Sundaram Thangavelu, \emph{Pointwise ergodic theorems for radial
averages on the {H}eisenberg group}, Adv. Math. \textbf{127} (1997), no.~2,
307--334.

\bibitem{PhongStein1991}
Duong~H. Phong and Elias~M. Stein, \emph{Radon transforms and torsion},
Internat. Math. Res. Notices (1991), no.~4, 49--60. \MR{1121165}

\bibitem{RoosSeeger}
Joris Roos and Andreas Seeger, \emph{Spherical maximal functions and fractal
dimensions of dilation sets}, arXiv:2004.00984, to appear in Amer. J. Math.,
2020.

\bibitem{RoosSeegerSrivastava}
Joris Roos, Andreas Seeger, and Rajula Srivastava, \emph{Lebesgue space
estimates for spherical maximal functions on {H}eisenberg groups}, Int. Math.
Res. Not. IMRN \textbf{2022} (2022), 19222--19257.

\bibitem{SchlagSogge1997}
Wilhelm Schlag and Christopher~D. Sogge, \emph{Local smoothing estimates
related to the circular maximal theorem}, Math. Res. Lett. \textbf{4} (1997),
no.~1, 1--15. \MR{1432805}

\bibitem{SeegerWaingerWright1995}
Andreas Seeger, Stephen Wainger, and James Wright, \emph{Pointwise convergence
of spherical means}, Math. Proc. Cambridge Philos. Soc. \textbf{118} (1995),
no.~1, 115--124. \MR{1329463}

\bibitem{Srivastava}
Rajula Srivastava, \emph{On the {K}or\'anyi spherical maximal function on
{H}eisenberg groups}, arXiv:2204.00695, Math. Ann., to appear, 2022.

\bibitem{SteinBeijing}
Elias~M. Stein, \emph{Oscillatory integrals in {F}ourier analysis}, Beijing
lectures in harmonic analysis ({B}eijing, 1984), Ann. of Math. Stud., vol.
112, Princeton Univ. Press, Princeton, NJ, 1986, pp.~307--355. \MR{864375}

\bibitem{Stein-harmonic}
\bysame, \emph{Harmonic analysis: real-variable methods, orthogonality, and
oscillatory integrals}, Princeton Mathematical Series, vol.~43, Princeton
University Press, Princeton, NJ, 1993, With the assistance of Timothy S.
Murphy, Monographs in Harmonic Analysis, III. \MR{1232192}

\end{thebibliography}
\end{document}